\documentclass{amsart}

\usepackage{amssymb}
\usepackage{stmaryrd}
\usepackage{color}
\usepackage{alltt}
\usepackage{upquote}
\usepackage{algorithm}
\usepackage{algorithmic}
\usepackage{graphicx}
\usepackage{subfigure}
\usepackage{url}

\title{\textsc{Verified computations for hyperbolic 3-manifolds}}

\author[N.Hoffman]{Neil Hoffman}
\address{Department of Mathematics and Statistics, University of Melbourne,
Victoria 3010, Australia}
\email{nhoffman@ms.unimelb.edu.au}

\author[K.Ichihara]{Kazuhiro Ichihara}
\address{Department of Mathematics, 
College of Humanities and Sciences, Nihon University,
3-25-40 Sakurajosui, Setagaya-ku, Tokyo 156-8550, Japan}
\email{ichihara@math.chs.nihon-u.ac.jp}

\author[M.Kashiwagi]{Masahide Kashiwagi}
\address{Department of Applied Mathematics, Waseda University,
3-4-1 Okubo, Shinjuku, Tokyo 169-8555 Japan}
\email{kashi@waseda.jp}

\author[H.Masai]{Hidetoshi Masai}
\address{Department of Mathematical and Computing Sciences, Tokyo Institute of
Technology, O-okayama, Meguro-ku, Tokyo 152-8552 Japan}
\email{masai9@is.titech.ac.jp}
\urladdr{}

\author[S.Oishi]{Shin'ichi Oishi}
\address{Department of Applied Mathematics, Waseda University,
3-4-1 Okubo, Shinjuku, Tokyo 169-8555 Japan}
\email{oishi@waseda.jp}
\urladdr{}

\author[A.Takayasu]{Akitoshi Takayasu}
\address{Department of Applied Mathematics, Waseda University,
3-4-1 Okubo, Shinjuku, Tokyo 169-8555 Japan}
\email{takitoshi@aoni.waseda.jp}


\newtheorem{thm}{Theorem}[section]    
\theoremstyle{definition}
\newtheorem*{rem}{Remark}             

\begin{document}
\begin{abstract}
For a given cusped 3-manifold $M$ admitting an ideal triangulation, we describe 
a method to rigorously prove that either $M$ or a filling of $M$
admits a complete hyperbolic structure via verified computer calculations.
Central to our method are an implementation of interval arithmetic and Krawczyk's Test. These techniques 
represent an improvement over
existing algorithms as they are faster while accounting for error accumulation in a more direct and user friendly way.
\end{abstract}

\dedicatory{Dedicated to Professor Sadayoshi Kojima on the occasion of his 60th birthday.}

\maketitle
\tableofcontents
\section{Introduction}

The study of 3-dimensional manifolds, often abbreviated by \textit{3-manifolds}, 
starts with the seminal papers by H. Poincar\'{e} in 1984--1904. 
In the last of these six papers, 
he raised a question, which became the famous \textit{Poincar\'{e} Conjecture}. 
This has been one of the driving forces for (low-dimensional) topology, 
and a great amount of effort was spent in pursuit of a solution.  
After nearly a century later, in 2002--2003, 
G. Perelman finally reached to the end of the struggles 
by providing an affirmative answer to the \textit{Geometrization Conjecture}, 
an extended version to the Poincar\'{e} Conjecture 
in \cite{Perelman1, Perelman2, Perelman3}. 
See \cite{Milnor2006, Morgan2005} as detailed references for example. 

The Geometrization Conjecture was raised by W. Thurston, 
who brought about dramatic changes to the study of 3-manifolds. 
He introduced a geometric view in the sense of Klein to the 3-manifold theory, 
and, 
in particular, he incorporated an amazing application of non-Euclidean geometry, namely \textit{hyperbolic geometry}
into the study of 3-manifolds.

As is well-known, every closed 2-dimensional manifold, i.e., closed surface, 
admits a geometric structure, a complete Riemannian metric of constant sectional curvature. 
Surprisingly Thurston predicted a similar situtation for the 3-dimensional case \cite{Thurston1982}. 
That is, every compact orientable 3-manifold can be 
canonically decomposed by cutting along essential 2-spheres and tori 
into the pieces, each of which admits a locally homogeneous geometric structure. 
This is the Thurston's Geometrization Conjecture. 
Actually he showed that there are exactly EIGHT geometries (i.e., models of geometric structures) 
to be considered, 
and essentially gave a proof for the case of 3-manifolds containing decomposing tori. 

Among the eight geometries, six have been well studied and are generally understood, 
the Seifert fibered geometries. 
The seventh is \textit{sol}-geometry. All manifolds admitting a sol geometric structure are torus bundles over the circle, 
or $n$-fold quotients of torus bundles over the circle where $n=2,4$. 
By most accounts, the most common and yet most interesting geometry occurs if a manifold $M$ admits a hyperbolic structure, 
i.e. $M$ admits a complete Riemannian metric of constant sectional curvature $-1$ of finite volume. 
See \cite{Milnor1982} for a good survey. 

In this paper, we describe a computer-aided practical method 
to rigorously prove that a given 3-manifold 
admits a complete hyperbolic structure via verified calculations. 
This procedure is already implemented as a software and available at \cite{Waseda}.

In the following, we give a sketch of the contents with the organization of the paper. 
In the discussions that follow, we will consider only 3-manifolds that are orientable, compact with boundary consisting of the (possibly empty) disjoint union of tori.


Our program takes in as an initial input
the combinatorial data of an ideal triangulation 
for the compact bounded case 
and a surgery description for the closed or partially filled case. 
From that data, as Thurston first described, we can find some algebraic equations, 
called a system of \textit{Gluing equations} with complex variables. 
If this system of equations has a system of complex solutions with positive imaginary parts, 
then the given manifold admits a hyperbolic structure. 
In the next section, we will recall a concise explanation for that, 
mainly based on the well-known book \cite{BenedettiPetronio1992}. 
Note that this process has already been implemented as \textit{SnapPea} by J. Weeks. 
Using the kernel code by Weeks, M. Culler and N. Dunfield implemented \textit{SnapPy} which is roughly speaking, a SnapPea on python.
(For further background on SnapPy or the SnapPea kernel see the SnapPy documentation and \cite{SnapPy}).
By abusing notation, in this paper, when we say SnapPea, we mean the Weeks's kernel code of SnapPea, {\em or} 
the Weeks's version of SnapPea. 
By SnapPy, we mean the actual program available at \cite{SnapPy}.
SnapPea uses the Newton's method to solve the equations, and hence, 
the solutions are just approximated ones. 
These approximated solutions, in principle, would not prove the convergence of the Newton's method.
Thus although SnapPea is very practical, it cannot give any rigorous mathematical proof 
for a given manifold to be hyperbolic. 

To prove that a given system of equations actually have a desired solution, 
we employ interval arithmetic and Krawczyk's test, both of which are explained in \S 3.
We will give a brief review of how to obtain mathematically rigorous conclusions from results of numerical computations.

In \S 4, the actual program {\em hikmot} and its implementation will be explained.
We further explain how to use it to rigorously prove the hyperbolicity of a given triangulated manifold.

The last two sections provide some of conclusions of our work, 
open problems, and expected future work. 
Also we will explain one application of our work 
to the study of exceptional surgeries on alternating knots,
which is a joint work of a part of the authors; Ichihara and Masai. This is followed by an appendix explaining the argument function.

\section*{Acknowledgements}

The authors would like to thank Mark Bell, Marc Culler, Nathan Dunfield, Craig Hodgson, Sadayoshi Kojima, and Bruno Martelli
 for a number of helpful conversations. 
 N.~H would like to thank Nihon University and the Tokyo Institute of Technology for 
graciously hosting him during a visit to Japan in the early stages of this project. 
Finally, for the remainder of this project,  N.~H has been supported by the Max Planck Institute for Mathematics and 
Australian Research Council Discovery Grant DP130103694.
K.~I was partially supported by
Grant-in-Aid for Young Scientists (B), No.~23740061, 
Ministry of Education, Culture, Sports, Science and Technology, Japan.
The work of H.~M was supported by JSPS Research Fellowship for Young Scientists.
S.~O and A.~T were partialy supported by the CREST, JST project tilted Establishment of Foundations of Verified Numerical Computations for Nonlinear Systems and Error-free Algorithms in Computational Engineering.

Finally, the authors wish to thank Tokyo Institute of Technology, the University of Texas Math Department, and Waseda University for providing computer support for this project.

\section{Gluing equations}
This section establishes our much notation for the rest of our paper. 
For the most part the notation is consistent the documentation for \cite{SnapPy} to ease the reading of this paper's companion computer code.  For further background on the hyperbolic geometry used in this paper, we refer the reader to  \cite{BenedettiPetronio1992} and \cite{ThurstonLectureNotes}. For some details on the various methods used by SnapPy, 
or 
the SnapPea Kernel to construct the necessary system of equations, we refer the reader to \cite{WeeksKnotTheory} and, of course, the documentation for SnapPy \cite{SnapPy}.

As noted in the introduction, we will assume our manifolds are orientable. We point out that our algorithm can verify the hyperbolicity of a non-orientable manifold $Q$ by establishing the hyperbolicity of the orientable double cover of $Q$. 

Using the upper half space as our model for hyperbolic 3-space, $\mathbb{H}^3$, we can identify the ideal boundary $\partial \mathbb{H}^3$ with $\mathbb{C} \cup \{\infty\}$. Also, we denote by $\overline{\mathbb{H}^3}$ the union $\mathbb{H}^3 \cup \partial \mathbb{H}^3$. The group of orientation preserving isometries acting on $\mathbb{H}^3$ is identified with $PSL(2,\mathbb{C})$ in the standard way. 

Refining our definition of the introduction, a 3-manifold $M$ is said to admit a hyperbolic structure
 if $M \cong \mathbb{H}^3/\Gamma$ where $\Gamma \subset PSL(2,\mathbb{C})$ is discrete 
 and the integral of the volume form for $\mathbb{H}^3$ over the fundamental domain
 for the quotient $\mathbb{H}^3/\Gamma$ is finite. For ease of notation, when $
 M$ admits a hyperbolic structure we will often just consider it as the
  quotient $\mathbb{H}^3/\Gamma$. 

An \emph{ideal tetrahedron} $T$ is a tetrahedron in $\overline{\mathbb{H}^3}$ with all 
four vertices  $\{z_1,z_2,z_3,z_4 \} \subset \partial \mathbb{H}^3$. 
There is an 
isometry $\gamma$ of $\mathbb{H}^3$ such that $\gamma(T)$ has vertices 
$\{0,1,\infty, z \}$, such that $\{\infty, 0,1, z \}$ and $\{z_1,z_2,z_3,z_4 \}$ 
have the same cross ratio. Under this definition, $z$ is only well defined up 
to the choice of which ordered subset of points we send to the ordered set 
$0,1,\infty$. However, $z$ is defined up to $z, \frac{z-1}{z}$ and $\frac{1}{1-z}$. 
 Furthermore, we can label each edge in our ideal tetrahedron by $z, 
\frac{z-1}{z}$ and $\frac{1}{1-z}$ such that the argument of the complex 
number is the dihedral angle along that edge (see Figure \ref
{subfig:IdealTetra01z}). By convention, 
 we will choose a complex argument function, $\arg(w)$, such that the range is  $(-\pi,\pi]$ for $w \in \mathbb{C}\setminus \{0\}$. 
 Also, we define $\log z = \log |z| + i \arg(z)$.
If $w$ is any complex parameter associated
 to the edge of an ideal tetrahedron with $\arg(w)<0$, we say 
the corresponding ideal tetrahedron is negatively oriented and if $\arg(w)>0$, we say the corresponding ideal tetrahedron is positively oriented. 
Ultimately, the algorithm described by this paper certifies when there is a solution to the Gluing equations such that all tetrahedral
parameters are positively oriented.

A \emph{truncated tetrahedron} is constructed by removing a neighborhood of the ideal
point (see Figure \ref{subfig:TruncatedTetra}). Technically, a cusped manifold (e.g a knot complement) 
is triangulated
by ideal tetrahedra and compact manifold with non-empty toroidal boundary (e.g a knot exterior) is triangulated by
truncated tetrahedra. Truncated tetrahedra have the property that opposite dihedral angles are equal as well
as the property that all angles on a triangular face sum to $\pi$ (the product of the complex parameters is $-1$). 

We now begin the process of describing the gluing equations
 associated to a triangulation $\mathcal{T}$ of a  hyperbolic 3-manifold $M$. Ultimately, 
there will be $n+2k+h$ equations, where $n$ is the 
number of tetrahedra in $\mathcal{T}$, $k$ in the number of unfilled 
cusps of $M$ and $h$ is the number of filled cusps of $M$. 
For convenience, we will use the parameter $m$ throughout this section to denote the $m$th equation 
in the set of $n+2k+h$ equations we construct in the following paragraphs. We will first assume that $h=0$ and then extend to the general case
in $\S$\ref{subsect:FilledCusps}. In $\S$\ref{subsect:FullSet}, we give a complete set of $n+2k+h$ equations before 
finally providing the set of $n$ equations that are solved by our computer algorithm.

\subsection{Edge equations} Each edge $e_m$ in a triangulation $\mathcal{T}$ of $M$, is an equivalence class of 
edges in the set of $n$ tetrahedra that make up $\mathcal{T}$. 
In this disjoint union of tetrahedra, we fix a labeling of the edges in the 
$j$th tetrahedron by the 
complex parameters: $z_j, \frac{z_j-1}{z_j}$ or $\frac{1}{1-z_j}$ such that
the opposite dihedral angles get the same label.    
By Euler characteristic conditions, there are $n$ equivalence 
classes of edges in $\mathcal{T}$, determined by the gluing 
identification. We can record by $a_{j,m}$
the number of times an edge labeled by
 $z_j$ is a member of the equivalence 
 class $e_m$. Similarly, we can define $b_{j,m}$
 as the number of times an edge labeled by
 $\frac{z_j-1}{z_j}$ is part of the
 equivalence class $e_m$ and finally, $c_{j,m}$
 as the number of times an edge labeled by 
 $\frac{1}{1-z_j}$ is part of the
 equivalence class $e_m$. If $\hat{\mathcal{T}}$ is 
a tessellation of $\mathbb{H}^3$ with $e_m$ an edge 
between $0$ and $\infty$, then
 the product of all the complex parameters around $e_m$
 is $1$ and the sum of their arguments is $2\pi$. In fact,
 for any edge $e_m$ there is an isometry of $\mathbb{H}^3$
 that sends the endpoints of $e_m$ to $0$ and $\infty$.
 Thus, around the $m$th edge, we get the equations:

\begin{align}
\prod_{j=1} ^n \left(z_j\right)^{a_{j,m}}\left(\frac{1}{1-z_j}\right)^{b_{j,m}}\left(\frac{z_j-1}{z_j}\right)^{c_{j,m}}=1, 
\end{align}
 and 
 \begin{align}
  \sum_{j=1} ^n {a_{j,m}} \cdot \arg\left(z_j\right)+ {b_{j,m}}\left(\frac{1}{1-z_j}\right)\cdot \arg+ {c_{j,m}}\cdot \arg\left(\frac{z_j-1}{z_j}\right) = 2\pi i.
\end{align} 
 
Note for a fixed $j$, $\sum_{\ell=1} ^n a_{j,\ell}=2$, $\sum_{\ell=1} ^n b_{j,\ell}=2$, 
and $\sum_{\ell=1} ^n c_{j,\ell}=2$, and since each of these coefficients is non-negative,
each $a_{j,m},b_{j,m}$, or $c_{j,m}$ is in $\{0,1,2 \}$. 

The notation above can also be presented in a more standard way using logarithms:

\begin{align}
\sum_{j=1} ^n a_{j,m} \log\left(z_j \right)  + b_{j,m} \log\left(\frac{1}{1-z_j}\right) +c_{j,m} \log\left(\frac{z_j-1}{z_j}\right) = 0+ 2\pi i .
\end{align}

\begin{figure}
        
        \begin{subfigure}
                \centering
                \resizebox{7 cm}{!}{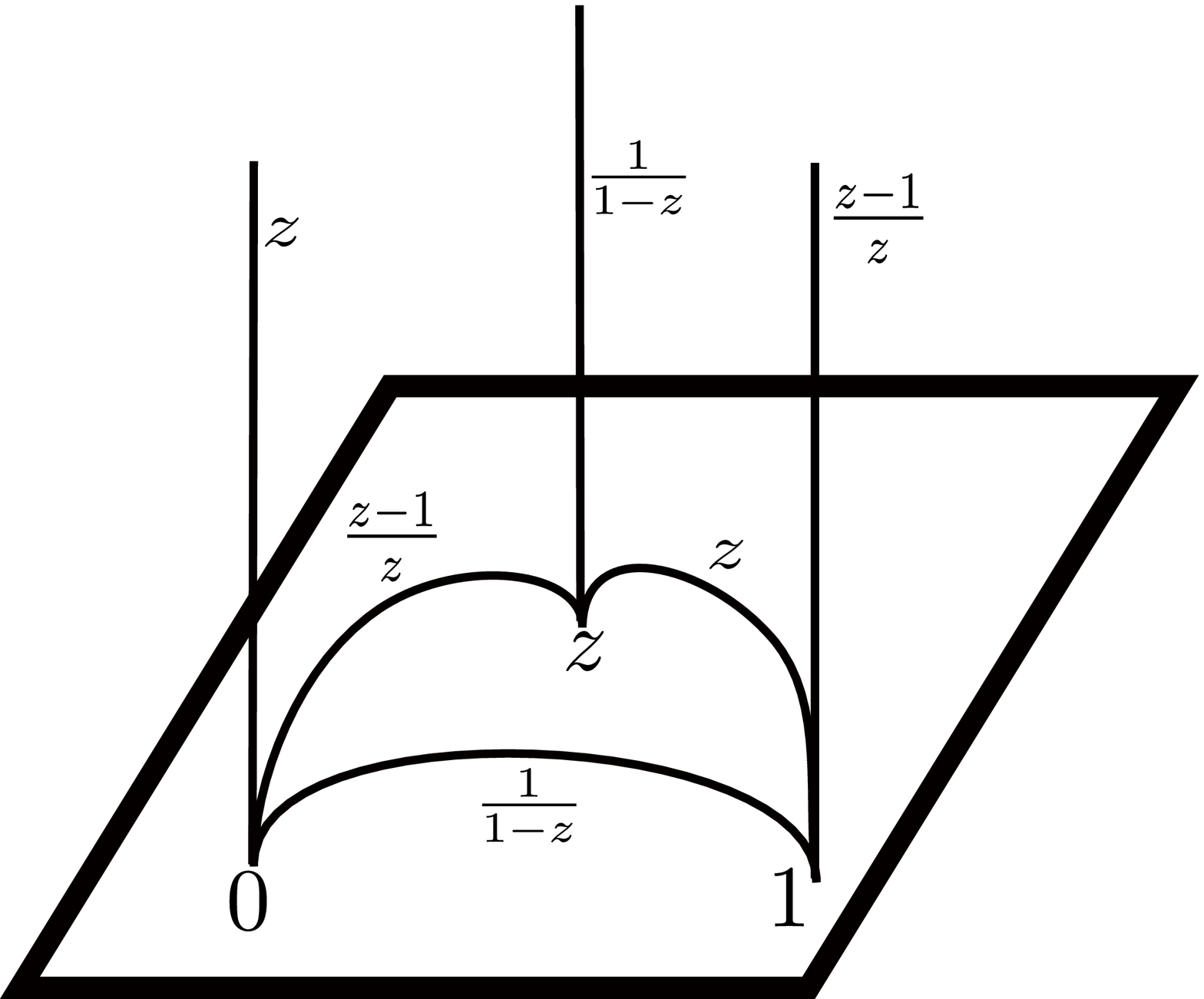}
                \caption{ \label{subfig:IdealTetra01z} An ideal tetrahedron with vertices at $0$,$1$, $\infty$, and $z$. 
  }
        \end{subfigure}%
        
        \begin{subfigure}
                \centering
                \resizebox{7 cm}{!}{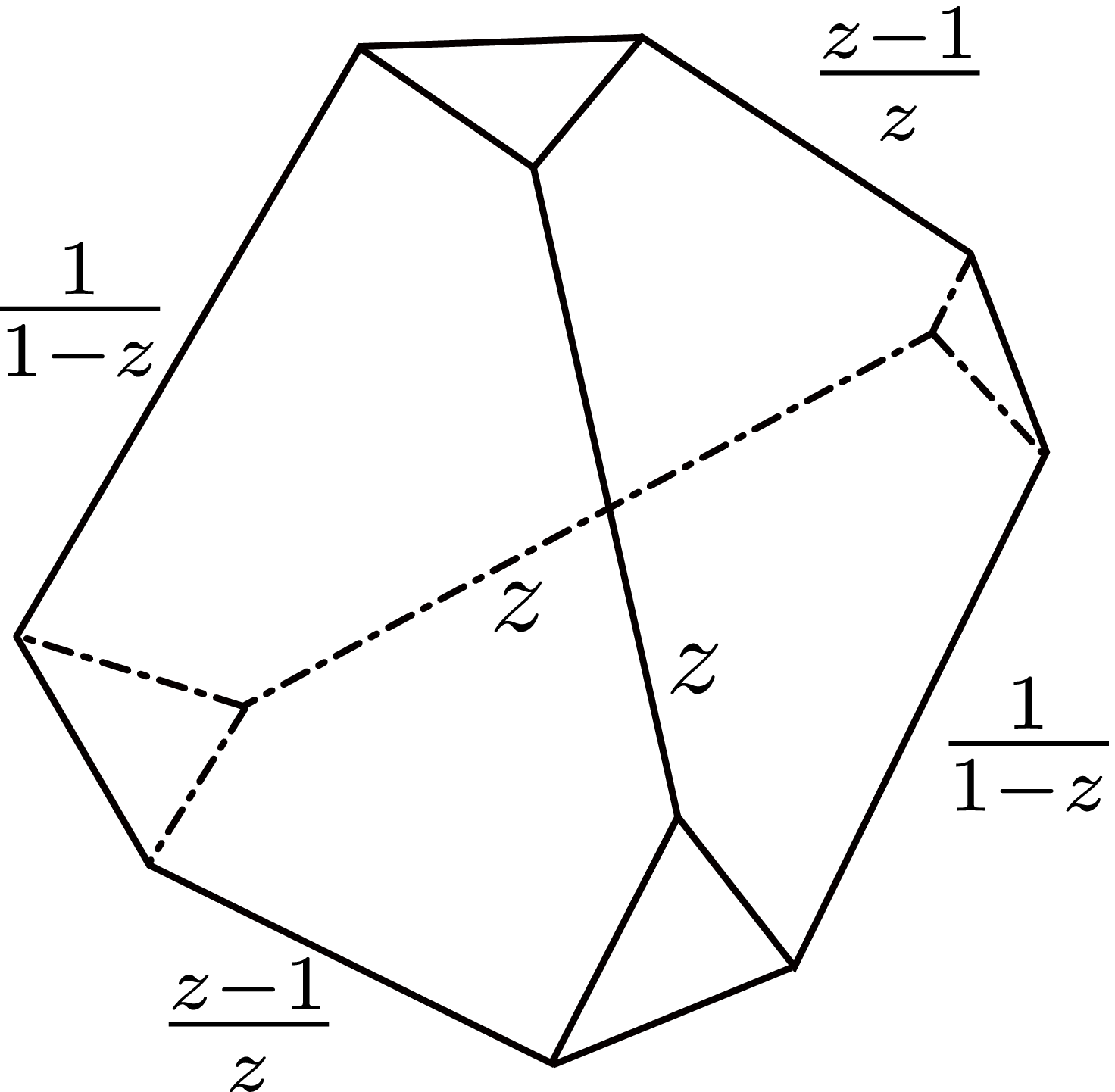}
                \caption{ \label{subfig:TruncatedTetra} A truncated tetrahedron with labeled edges}
        \end{subfigure}%

\end{figure}


\subsection{Unfilled cusp equations} Although this sets up a system of $n$ equations with $n$ unknowns, this system does not specify a unique hyperbolic structure (see \cite[Theorem 5.6]{ThurstonLectureNotes} for example). Instead, we must add 
more equations corresponding to the $k$ cusps of $M$. The construction is as follows. For the $t$th cusp of $M$ there is a peripheral torus $C_t$. 
When  viewed from a cusp, any fundamental domain for $C_t$ can be identified with a quotient of the plane by two translations $\mu_t, \lambda_{t}$. 
Intersecting this plane, with the 2-skeleton of $\hat{\mathcal{T}}$ results the plane tessellated by triangles. In the 1-skeleton of this tessellation, $\mu_t, \lambda_{t}$ lift to two oriented 
piece-wise linear 
curves, which we denote by $\gamma_{t,m}$ and $\gamma_{t,l}$, respectively (see Figure \ref{fig:MuLambda}). 
 The following conditions ensure that $C_t$ is the quotient having a Euclidean structure X by a group of translations generated by $\mu_t$ and $\lambda_t$. 
At each vertex $v$ along the path $\gamma_{t,m}$ or $\gamma_{t,l}$, we can take the product of 
dihedral angles to the right side of the oriented curve $p_v$. It is a consequence of 
\cite[Lemma E.6.8]{BenedettiPetronio1992}, that along a path containing $\#v$ vertices this condition is equivalent
to the product of all complex parameters the $p_v$ is $(-1)^{\#v}$, and so we build the $m$th equation corresonding to the meridian $\mu_t$ as follows.
For the $j$th tetrahedron, let $(-1)^{\#v} a_{j,m}$ be the number of times the complex parameter $z_j$ is part of this product, 
and let $(-1)^{\#v}b_{j,m}$ be the number of times the complex parameter $\frac{z_j-1}{z_j}$ shows up in $p_v$. 
Finally, let $(-1)^{\#v}c_{j,m}$ be the number
of times the complex parameter $\frac{1}{1-z_j}$ corresponds to an angle along this product.
This method yields the following $2k$ pairs of equations:

\begin{align}
\prod_{j=1} ^{n} \left(z_j\right)^{a_{j,m}}\left(\frac{1}{1-z_j}\right)^{b_{j,m}}\left(\frac{z_j-1}{z_j}\right)^{c_{j,m}}=1
\end{align}
 and 
 \begin{align}
  \sum_{j=1} ^{n} {a_{j,m}} \cdot \arg\left(z_j\right)+ {b_{j,m}}\cdot \arg\left(\frac{1}{1-z_j}\right)+ {c_{j,m}}\cdot \arg\left(\frac{z_j-1}{z_j}\right) = 0.
\end{align}  

The $m+1$st equation will follow the same recipe but we will replace the 
path determined by $\mu_t$ with the path determined by $\lambda_t$. 

 We note that SnapPea uses a slightly different 
method, however using identities such as $z_j \cdot \frac{1}{1-z_j}\cdot \frac{z_j-1}{z_j}=-1$
 the equations can be seen as equivalent. Furthermore, in the definition above
  the coefficients $a_{j,m}$, ${b_{j,m}}$, and ${c_{j,m}}$ would have the same
  sign and have absolute value less than 4, but under the identity
   $z_j \cdot \frac{1}{1-z_j}\cdot \frac{z_j-1}{z_j}=-1$, the sign of the coefficients can change. 
Just as above, both conditions can be satisfied if 
\begin{align}
\sum_{j=1} ^{n} a_{j,m} \log\left(z_j\right) + b_{j,m} \log\left(\frac{1}{1-z_j}\right)+c_
{j,m} \log\left(\frac{z_j-1}{z_j}\right) = 0+0\pi i.
\end{align}

Thus, since $M$ has $n$ tetrahedra and $k$ cusps (all unfilled), we have a system of 
$n+2k$ equations. 



\begin{figure}
\resizebox{9 cm}{!}{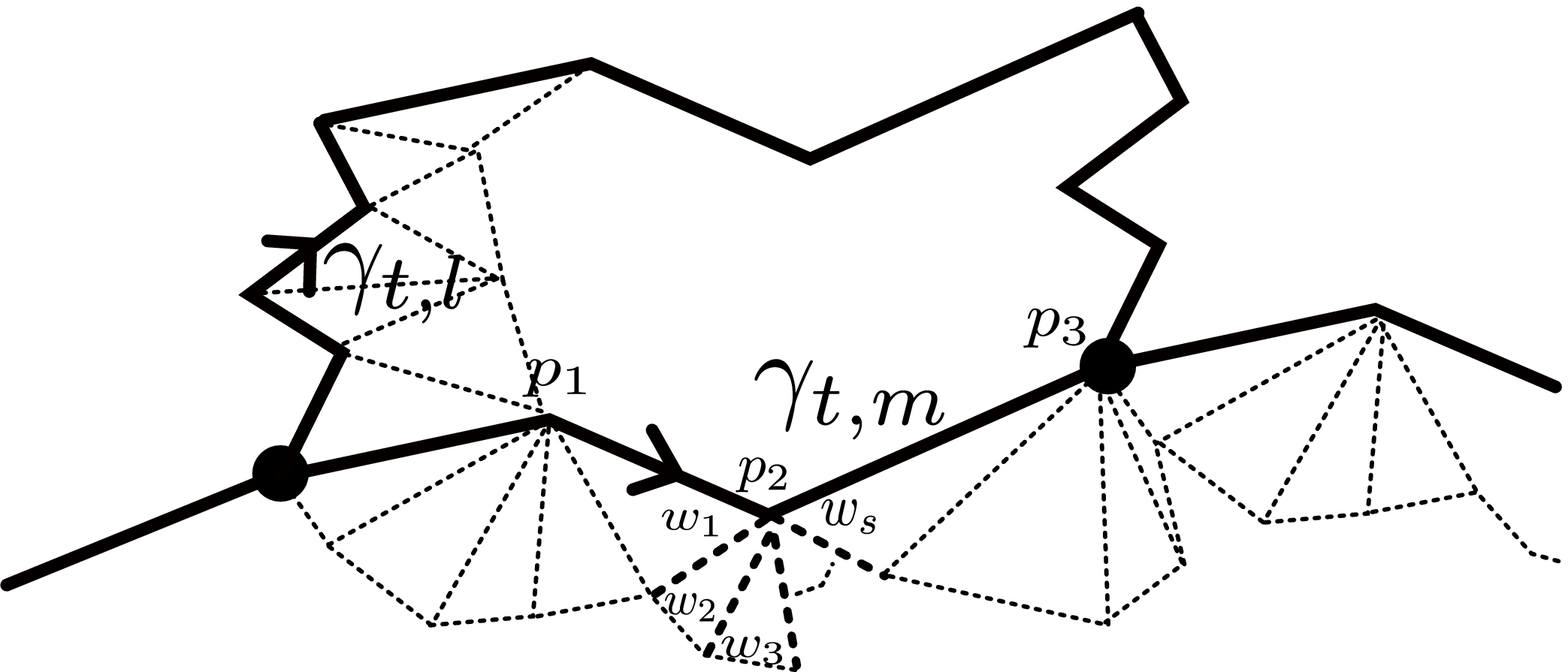}
\caption{\label{fig:MuLambda} The fundamental domain for $C_i$ with an identified boundary. In this example $\gamma_{t,m}$ is the path between the two dots and the product at point $p_2$ that contributes to the equation for $\gamma_{t,m}$ is $(-1)^3 w_1 ... w_s$, where $w_i$ denotes the complex parameter of an edge with endpoint at the cusp.}
\end{figure}

\subsection{Filled cusp equations}\label{subsect:FilledCusps} A cusped 3-manifold $M$ can be Dehn filled along a cusp $C_t \cong T^2 \times[0,\infty)$ if the set $T^2 \times(0,\infty)$ is removed and the boundary 
$T^2 \times \{0 \}$ is identified with the boundary of a solid torus. These identifications can be parametrized by the curve $\gamma$ in $\pi_1(C_t)= 
\langle\mu_{j,t}, \lambda_{j,t} | [\mu_{j,t}, \lambda_{j,t}] \rangle$ that gets identified with the curve that bounds a disk in the solid torus. If $M$ is hyperbolic, 
all but at most finitely many choices of $\gamma = p \mu_{j,t} + q \lambda_{j,t}$ will yield a hyperbolic manifold by Thurston's Hyperbolic Dehn Surgery 
Theorem \cite{ThurstonLectureNotes}. Furthermore, the process can be repeated if $M$ is filled along multiple cusps.

Denote by $N$ a manifold that comes from filling $h$ cusps of a manifold $M$. 
Using the convention that $M$ decomposes into $n$ tetrahedra and has $k$ cusps, 
we will assume that the system of $n+2k$ equations described above is already constructed. 

Denote by $\Omega_M$  the set of gluing equations for $M$. For each filled cusp in $N$, there is of course
a corresponding unfilled cusp of $M$, and so assume that $r$th and $r+1$st equations in $\Omega_M$ correspond
to the meridian and longitude (respectively) of this cusp in $M$. For the equation set of $N$, we replace these two equations
by (a simplified form of) the following:

\begin{align}
\label{Equation.Shapes} \prod_{j=1}^n & \left(\left(z_j\right)^{a_{j,r}}\left(\frac{1}{1-z_j}\right)^{b_{j,r}}\left(\frac{z_j-1}{z_j}\right)^{c_{j,r}} \right)^p \nonumber \\ & \cdot \left(\left(z_{j}\right)^{a_{j,r+1}}\left(\frac{1}{1-z_j}\right)^{b_{j,r+1}}\left(\frac{z_j-1}{z_j}\right)^{c_{j,r+1}}\right)^q =1 
\end{align}

and similarly for the argument functions:

\begin{align}
\label{Equation.Argument} & p\cdot \sum_{j=1} ^n \left( {a_{j,r}} \cdot \arg\left(z_j\right)+ {b_{j,r}}\cdot \arg\left(\frac{1}{1-z_j}\right)+ {c_{j,r}}\cdot \arg\left(\frac{z_j-1}{z_j}\right)\right)  \\
&+ q\cdot \sum_{j=1} ^{n} \left( {a_{j,r+1}} \cdot \arg\left(z_j\right)+ {b_{j,r+1}}\cdot \arg\left(\frac{1}{1-z_j}\right)+ {c_{j,r+1}}\cdot \arg\left(\frac{z_j-1}{z_j}\right)\right) =2\pi. \nonumber
\end{align} 

We can also express these equations using logarithms.

\begin{align}
p \sum_{j=1} ^{n} a_{j,r}& \log\left(z_j\right) + b_{j,r} \log\left(\frac{1}{1-z_j}\right)+c_{j,r} \log\left(\frac{z_j-1}{z_j}\right) + \notag \\
& q \sum_{j=1} ^n a_{j,r+1} \log\left(z_j\right) + b_{j,r+1} \log\left(\frac{1}{1-z_j}\right)+c_{j,r+1} \log\left(\frac{z_j-1}{z_j}\right) = 0 + 2 \pi i.
\end{align}

Again, to show that a given triangulation of $N$ corresponds to a hyperbolic structure, the solution set to these equations must be $0$ dimensional and 
have at least one solution corresponding to all 
positively oriented tetrahedra. However, in the case where $N$ has filled cusps, the solution corresponds to an incomplete hyperbolic structure. We refer the reader to \cite[$\S$E.6-iv]{BenedettiPetronio1992} for the details of how to 
extend this to a complete hyperbolic structure. 

\subsection{Full set of equations}\label{subsect:FullSet} In the paragraphs above, there are two equivalent sets of equations. The equations 
with polynomials and arguments we will call the \emph{rectangular 
equations} and the equations with logarithms we will call the \emph{$\log$ equations}. 
For a manifold $N$ with $n$ tetrahedra, $k$ unfilled cusps and $h$ filled 
cusps, then there are $n+2k+h$ $\log$ equations needed to be solved and 
$n+2k+h$ pairs of polynomial and argument equations. It is common to combine 
the features of both of these equations in a often more convenient form by 
removing the third term in each summand of the $\log$ equations (see \cite
[page 235]{BenedettiPetronio1992} for example). 

\begin{align}\label{Equation.CompleteSetLog}
\left\{\sum_{j=1} ^{n} ( a_{j,m} -c_{j,m}) \log\left(z_j\right) + (-b_{j,m}+c_{j,m}) \log\left(1-z_j\right) + c_{j,m} \cdot \pi i = \epsilon_m \right\}_{m=1} ^{n+2k+h}.
\end{align}

Here, the constant term $\epsilon_m$ is $0$ if the equation corresponds to an unfilled cusp and $2\pi i$ otherwise. 
Finally, Equation \ref{Equation.CompleteSetLog} can be expressed as polynomial and argument equations as follows:

\begin{align}\label{Equation.CompleteSetPolynomial}
& \left\{\prod_{j=1} ^{n}  (z_j)^{( a_{j,m} -c_{j,m})}\cdot (1-z_j)^{(-b_{j,m}+c_{j,m})}= \prod_{j=1} ^{n} (-1)^{c_{j,m}} \right\}_{m=1}^{n+2k+h} \mbox{ and } \\
&\left\{ \sum_{j=1} ^{n} \arg((z_j)^{( a_{j,m} -c_{j,m})}) + \arg( (1-z_j)^{(-b_{j,m}+c_{j,m})})  = \epsilon_m-  \sum_{j=1} ^{n} c_{j,m} \cdot \pi i \right\}_{m=1}^{n+2k+h} \nonumber
\end{align}

\noindent where just as above, we choose $\epsilon_m$ to be $2\pi i$ for equations corresponding to edges or filled cusps and $0$ in the case of an unfilled cusp. This structure is unique by the following result of Thurston. It is essentially a 
direct consequence of \cite[Corollary 5.7.3]{ThurstonLectureNotes}, however, we present as follows to better fit the 
discussion of this section.


\begin{thm}[Thurston]
If the gluing equations admit a solution, each of whose imaginary part is positive (positive solution, in short),
then M admits a complete hyperbolic structure of finite volume.
Furthermore the structure is determined by the solution and 
is unique up to homeomorphism and re-triangulation of the manifold.  
\end{thm}
 Also, to ease the notation we define $\alpha_{j,m}=a_{j,m} -c_{j,m}$ and $\beta_{j,m}= -b_{j,m} +c_{j,m}$. Furthermore, we can define 
$(n+2k+h) \times 2n$ matrix  ${\Lambda_M}$ as follows:

\begin{align}\label{matrix.CoeffientList}
{\Lambda_M} = \left[
	\begin{array}{cccccc}
		\alpha_{1,1}&\cdots&\alpha_{n,1}&\beta_{1,1}&\cdots&\beta_{n,1}\\
		\vdots&\ddots&\vdots&\vdots&\ddots&\vdots\\
		\alpha_{1,n+2k+h}&\cdots&\alpha_{n,n+2k+h}&\beta_{1,n+2k+h}&\cdots&\beta_{n,n+2k+h}
	\end{array} \right].
\end{align}

The numbering of the equations in this section most closely describes the numbering convention of SnapPea. However, in the computer argument
that follows, we start by changing the order of the equations such that for a manifold $M$ with $n$ tetrahedra, $k$ unfilled cusps, and  $h$ filled cusps,
equations $n$, such that the first $n$ equations correspond to edge equations, the next $2k$ equations correspond to unfilled cusps and the final $h$ 
equations correspond to filled cusps. Namely, we can make a new matrix
 $\widehat{\Lambda_M}$ such that the first $k+h$ rows are selected from rows of $
\Lambda_M$ corresponding to either filled cusps or meridian equations of unfilled cusps and then choose a set of $n-(k+h)$ equations such that the 
corresponding rows of $\Lambda_M$ have maximal rank, finally denote by $S$, the total set of rows of selected to form $\widehat{\Lambda_M}$. 
The following theorem of Neumann and Zagier \cite{NeumannZagier} shows that if the gluing equations 
admit a positive solution, then $\widehat{\Lambda_M}$ has rank $n$
(see also \cite[Lemma 2.4]{HMoser} and \cite[E.6.ii-iv]{BenedettiPetronio1992}). 

\begin{thm}[Neumann \& Zagier \cite{NeumannZagier}]\label{Theorem:2_1}
If $M$ and the associated triangulation
correspond to a hyperbolic structure,
$\widehat{\Lambda_M}$ has rank $n$. Furthermore,
in this case, solutions to the following equations 

\begin{align}\label{Equation.EfficientSetPolynomial}
& \left\{\prod_{j=1} ^{n}  (z_j)^{ \alpha_{j,m}}\cdot (1-z_j)^{\beta_{j,m}}=  \gamma_m =\prod_{j=1} ^{n} (-1)^{c_{j,m}} \right\}_{m\in S} \mbox{ and } \\
&\left\{ \sum_{j=1} ^{n} \arg((z_j)^{ \alpha_{j,m}}) + \arg( (1-z_j)^{\beta_{j,m}})  = \epsilon_m-  \sum_{j=1} ^{n} c_{j,m} \cdot \pi i \right\}_{m\in S} \nonumber
\end{align}

\noindent form a zero dimensional algebraic set. 
\end{thm}

As it suffices to solve only the equations (both polynomial and argument) of in the theorem above, the remainder of this paper will be dedicated to verifying that a solution to the system of equations in (\ref{Equation.EfficientSetPolynomial}) can be obtained from a suitable approximation. 
Also, we have taken care to define the right hand side of (\ref{Equation.EfficientSetPolynomial}), especially with regard to the right hand side of the argument equations. We make the following remark to justify this attention. 

\begin{rem}
By \cite[Lemma E.6.3]{BenedettiPetronio1992}, there is no need to check the argument condition for any equation corresponding to an edge. However, this condition must be checked for any equation coming from a filled or unfilled cusp. For example, the complete solution to the equations for the census manifold $m007$,  is also a solution to the polynomial part of the rectangular equations for $m007(3,1)$, but the argument condition of the equation corresponding to the filled cusp is not satisfied by that solution.
\end{rem}


\section{Krawczyk's Test}\label{sec.Krawczyk}
In this section, we present an algorithm which performs a mathematically rigorous numerical existence test for solutions of gluing equations (\ref{Equation.CompleteSetPolynomial}).
In such an algorithm, one must deal with all errors in numerical computation including truncation errors and rounding errors.
In usual numerical computations, one uses the floating-point arithmetic.
The floating-point arithmetic approximates the real and complex arithmetic.
It is well-known that it may occur a rounding error in every floating-point operation.
Thus,
it is necessary to discuss how to obtain mathematically rigorous conclusions from results of numerical computations based on the floating-point arithmetic.
For that purpose, the idea of interval arithmetics is useful.
Historically, the concept of interval arithmetic has been proposed in 1950's \cite{Sunaga1,Sunaga2,Moore}.
Since then, a lot of work has been done in this area (see \cite{Markov_Okumura,Moore,Rump} for example).
In the following, we first review how to implement the interval arithmetic by using the floating-point arithmetic.
In the end of this section, a simple example treating the gluing equation is shown for demonstration.

\subsection{Brief Review of Interval Analysis}\label{sec.Interval}
Let us write an interval on the set of real numbers, $\mathbb{R}$, as
$X:=\{x\in\mathbb{R}:\underline{x}\le x\le\overline{x},~\underline{x}, \overline{x}\in\mathbb{R}\}=[\underline{x}, \overline{x}]$.
We denote the set of such intervals by $\mathbb{I}\mathbb{R}$.
The midpoint of the interval: $\mathrm{mid}(X)\in\mathbb{R}$ is defined by
\[
	\mathrm{mid}(X):=\frac{\overline{x}+\underline{x}}{2}.
\]
The radius of the interval $X$ is defined by
\[
	\mathrm{rad}(X):=\frac{\overline{x}-\underline{x}}{2}.
\]
Let us assume that a unary operation $u:\mathbb{R}\to\mathbb{R}$ and a binary operation
$\circ:\mathbb{R}\times\mathbb{R}\to\mathbb{R}$ are defined.
Most unary operations we consider here are standard functions.
Then, we can extend these operations to
an unary operation $\tilde u:\mathbb{I}\mathbb{R}\to\mathbb{I}\mathbb{R}$ and
a binary operation $\tilde \circ:\mathbb{I}\mathbb{R}\times\mathbb{I}\mathbb{R}\to\mathbb{I}\mathbb{R}$,
including basic four arithmetic operations $\circ\in\{+,-,\cdot,/\}$, by
\begin{equation}\label{unary}
	\tilde u(X):=\{u(x):x\in X\}
\end{equation}
and
\begin{equation}\label{binary}
	X\,\tilde\circ\,Y:=\{x\circ y:x\in X,y\in Y\},
\end{equation}
respectively.
Here, $X$, $Y\in\mathbb{I}\mathbb{R}$.
Although in order to calculate four basic operations $\tilde\circ\in\{+,-,\cdot,/\}$, which are called the interval arithmetics,
it seems to need infinitely many computations,
only finitely many calculations of end points of intervals are sufficient.
Namely, the interval arithmetic can be
executed by
\begin{align*}
	X+Y&=\left[\underline{x}+\underline{y},\overline{x}+\overline{y}\right]\\
	X-Y&=\left[\underline{x}-\overline{y},\overline{x}-\underline{y}\right]\\
	X\cdot Y&=\left[\min\{\underline{x}\cdot\underline{y},
	\overline{x}\cdot\overline{y},
	\underline{x}\cdot\overline{y},
	\overline{x}\cdot\underline{y}\},
	\max\{\underline{x}\cdot\underline{y},
	\overline{x}\cdot\overline{y},
	\underline{x}\cdot\overline{y},
	\overline{x}\cdot\underline{y}\}\right]\\
	X/Y&=X\cdot\left[\frac{1}{\overline{y}},\frac{1}{\underline{y}}\right],~(0\not\in Y)
\end{align*}
for $X=[\underline{x},\overline{x}]$ and $Y=[\underline{y},\overline{y}]$.

An interval vector is defined as $m$-tuple of interval entries satisfying
\[
	X:=\{x\in\mathbb{R}^m:x_i\in X_i~\mbox{for}~1\le i\le m\}
\]
for $X_i\in\mathbb{I}\mathbb{R}$,
which is the Cartesian product of one-dimensional intervals.
The set of interval vectors is denoted by $\mathbb{I}\mathbb{R}^m$.
The set of interval matrices $\mathbb{I}\mathbb{R}^{m\times m}$ is defined similarly.
The midpoint and the radius of an interval vector or an interval matrix are also defined component-wisely.
For example, letting $A\in\mathbb{I}\mathbb{R}^{m\times n}$ be an interval matrix,
the $\mathrm{mid}(A)$ is defined by
\[
	\mathrm{mid}(A)=\left[
	\begin{array}{ccc}
		\mathrm{mid}(a_{11})&\cdots&\mathrm{mid}(a_{1n})\\
		\vdots&\ddots&\vdots\\
		\mathrm{mid}(a_{m1})&\cdots&\mathrm{mid}(a_{mn})
	\end{array}
	\right]\in\mathbb{R}^{m\times n}
\]
for $a_{ij}\in\mathbb{I}\mathbb{R}$.


Now we shall discuss how to solve the system of gluing equations (\ref{Equation.CompleteSetPolynomial}) via interval arithmetic.
The system (\ref{Equation.CompleteSetPolynomial}) can be rewritten as
\[
	\prod_{j=1}^n z_j^{\alpha_{j,m}}(1-z_j)^{\beta_{j,m}}=\gamma_{m},~m=1,\cdots, n.
\]
Note that as discussed in \S \ref{subsect:FullSet},  $\alpha_{j,m},\beta_{j,m},\gamma_m\in\mathbb{Z}$.
Theorem \ref{Theorem:2_1} states that
the system of equations (\ref{Equation.CompleteSetPolynomial}) has at most $n$-independent equations.
To compute the matrix $S$ in Theorem \ref{Theorem:2_1}, we need to compute the ranks of matrices.
Note that if we compute the rank naively, it involves a lot of calculations and hence 
it is not very easy to rigorously compute the rank.
An efficient way to conjecture the rank of a matrix $A$ is to use the singular value decomposition of $A$
 because the rank is equal to the number of non-zero singular values.
In actual numerical computation,
one can only have approximated singular values. 
Thus, we set a threshold $\delta>0$, which on the order of $10^{-8}$.
Then, if all singular values are greater than $\delta$, then we can expect that a given matrix has full rank.
We can thus compute a candidate of $S$ by choosing $n-k$ rows of $\Lambda_M$ 
one by one so that at each step the resulting matrix has full rank.
In the event that our threshold was too big and we can not find any candidate,
we randomize the triangulation of $M$ with SnapPy and try to solve another system of gluing equations.

We assume here that we have successfully selected a candidate of $S$ and hence, 
$\widehat\Lambda_M\in\mathbb{R}^{n\times 2n}$ in Theorem \ref{Theorem:2_1}.



Setting $\alpha_{j,m}^+=\max\{\alpha_{j,m},0\}$, $\alpha_{j,m}^-=-\min\{\alpha_{j,m},0\}$,
$\beta_{j,m}^+=\max\{\beta_{j,m},0\}$, and $\beta_{j,m}^-=-\min\{\beta_{j,m},0\}$,
we can rewrite selected $n$ equations corresponding to $\widehat\Lambda_M$ as
\begin{equation}\label{orgequation}
	\prod_{j=1}^nz_j^{\alpha_{j,m}^+}(1-z_j)^{\beta_{j,m}^+}-\prod_{j=1}^nz_j^{\alpha_{j,m}^-}(1-z_j)^{\beta_{j,m}^-}\gamma_{k}=0,~k=1,\cdots, n.
\end{equation}
For $z=(z_1,\cdots,z_n)\in\mathbb{C}^n$, the gluing equation (\ref{orgequation}) can be rewritten as $g(z)=0$,
where $g$ is a mapping from $\mathbb{C}^{n}$ onto itself.
Setting $z_j=x_{2j-1}+ix_{2j}~(j=1,\cdots,n)$, we can rewrite (\ref{orgequation}) further as
\begin{equation}\label{considered_eq}
	f(x)=0,
\end{equation}
where $f:\mathbb{R}^{2n}\to\mathbb{R}^{2n}$ is a differentiable nonlinear real mapping.
Let us put $m=2n$.


Based on this preparation, we now discuss how to use the interval arithmetic to prove the existence and local uniqueness of a solution for (\ref{considered_eq}) in the interval vector $X\in\mathbb{I}\mathbb{R}^m$.
Let $c\in X$ be an approximation of a solution of $f(x)=0$.
We introduce a simplified Newton mapping $s:\mathbb{R}^m\to\mathbb{R}^m$ for the mapping $f$ by
\begin{equation}\label{snewton}
	s(x):=x-Rf(x),
\end{equation}
where $R\in\mathbb{R}^{m\times m}$ is a certain matrix.
Usually, $R$ is chosen to be an approximate inverse of $f'(c)$.
If $R$ is invertible, $f(x)=0$ and $s(x)=x$ becomes equivalent.
From the contraction mapping principle, if we can show $s$ to be a contraction mapping from $X$ into itself, 
say $s$ is a contraction on $X$, for short, then we can prove the existence and local uniqueness of solution for (\ref{considered_eq}) on $X$ provided that $R$  is invertible.
To check whether $s$ is a contraction on $X$, usually Newton-Kantorovich type theorem \cite{Kantorovich} is applied.
For the finite dimensional case, so-called Krawczyk test \cite{Krawczyk1,Krawczyk2,Rump} is often used, since in many cases its implementation is straightforward and easy compared with the direct application of Newton-Kantorovich type theorem.
For equation (\ref{considered_eq}), we will show it is the case.
So-called Krawczyk's mapping $K:\mathbb{I}\mathbb{R}^m\to\mathbb{I}\mathbb{R}^m$, 
a map proposed in \cite{Krawczyk1}, with respect to the gluing equation (\ref{considered_eq}) is defined by
\begin{eqnarray}\label{Equation.K}
	K(X)=c-Rf(c)+(I-Rf'(X))(X-c),
\end{eqnarray}
where $I\in\mathbb{R}^{m\times m}$ is a unit matrix.
We note that $K(X)$ is a mean value form of $s$.
If we consider $s(X)$ based on the naive interval extension,
which is just obtained by replacing the basic four arithmetic operation on reals with those of intervals,
it follows that 
$s(X)\supset X$ so that we cannot expect $s(X)\subset\mathrm{int}(X)$.
Here, $\mathrm{int}(X)=\{x=(x_1,\cdots,x_m)\in X:\underline{x}_i<x_i<\overline{x}_i~(i=1,\cdots,m)\}$.
On the contrary to this
\[
	K(X)\subset\mathrm{int}(X)
\]
can be expected.
Indeed, the following theorem holds, which states a sufficient condition that there exists a solution of (\ref{considered_eq})
which is unique in $X$.
On the basis of Krawczyk's mapping \cite{Krawczyk1}, Krawczyk's test is established by S.M. Rump \cite{Rump}.

\begin{thm}[Krawczyk's test]\label{Theorem.Krawczyk}
For a given interval $X\in\mathbb{I}\mathbb{R}^m$,
let $\mathrm{int}(X)$ be the interior of $X$.
If the condition
\begin{eqnarray}\label{condition.Krawczyk}
	K(X)\subset \mathrm{int}(X)
\end{eqnarray}
holds, then there uniquely exists an exact solution $x^*$ of (\ref{considered_eq}) in $X$.
Furthermore, 
it is also shown that
$R$ and all matrices $C\in f'(X)$ including $f'(x^*)$ are nonsingular.
\end{thm}

In the following, we discuss how to implement Krawczyk's test using the floating-point arithmetic.

\subsection{Machine Interval Arithmetic}\label{Subsection:interval}
While there exist many definitions of floating point systems, in order to simplify the discussion, we consider
 the floating-point number system obeying IEEE 754 standard \cite{IEEE754}.
Let $\mathbb{F}$ be a set of IEEE 754 double precision floating-point numbers.
In IEEE 754 standard, four rounding operations from $\mathbb{R}$ to $\mathbb{F}$ are defined.
The rounding to the nearest $\mathrm{fl}:\mathbb{R}\to\mathbb{F}$ is one of them. This is defined through
\[
	|\mathrm{fl}(x)-x|=\min\{|c-x|:c\in\mathbb{F}\}
\]
for $x\in\mathbb{R}$ with $|x|\le\max\{|a|:a\in\mathbb{F}\}$.
As well as rounding to nearest, IEEE 754 standard defines the rounding towards $-\infty$,
$\mathrm{fldown}:\mathbb{R}\to\mathbb{F}$, and the rounding towards $\infty$,
$\mathrm{flup}:\mathbb{R}\to\mathbb{F}$, by
\[
	\mathrm{fldown}(x):=\max\{c\in\mathbb{F}:c\le x\}
\]
and
\[
	\mathrm{flup}(x):=\min\{c\in\mathbb{F}:x\le c\},
\]
respectively.
Although we omit the explanation, IEEE 754 also defines rounding towards zero.
It is known that almost all CPUs including Intel processors are designed to satisfy IEEE 754 standards.

Let $\mathbb{I}\mathbb{F}\subset\mathbb{I}\mathbb{R}$ denote the set of intervals with floating-point end points: $\{[\underline{x},\overline{x}]:\underline{x},\overline{x}\in\mathbb{F}, \underline{x}\le\overline{x}\}$.
We define a rounding $\oblong:\mathbb{I}\mathbb{R}\to\mathbb{I}\mathbb{F}$ by
\[
	\oblong([\underline{x},\overline{x}]):=[\mathrm{fldown}(\underline{x}),\mathrm{flup}(\overline{x})].
\]
Using this, the basic four arithmetic operations of intervals $X,Y\in\mathbb{I}\mathbb{F}$, $\boxcircle:\mathbb{I}\mathbb{F}\times\mathbb{I}\mathbb{F}\to\mathbb{I}\mathbb{F}$, are defined by
\[
	X\boxcircle Y:=\oblong(X\circ Y),
\]
where $\circ\in\{+,-,\cdot,/\}$. This is called a machine interval arithmetic.
The interval vectors and matrices with floating-point end points
$\mathbb{I}\mathbb{F}^m$, $\mathbb{I}\mathbb{F}^{m\times m}$
are also the Cartesian product of one-dimensional intervals.

Although we introduce this notation to stress the operations the computer is actually perfroming, 
one of the advantages of interval arithmetic is that one can overwrite a
computer's calls for the basic operations of arithmetic with the
definitions above. Similarly, for the remainder of the paper, we will
suppress the notation $\boxcircle$ in favor of the less cumbersome
$\{+,-,\cdot, /\}$.


Since the gluing equation (\ref{considered_eq}) is based on the basic four arithmetic operations,
we can replace each arithmetic by its corresponding interval arithmetic.
Then, by this replacement, we can construct mappings $F:\mathbb{I}\mathbb{F}^m\to\mathbb{I}\mathbb{F}^m$ and
$F':\mathbb{I}\mathbb{F}^m\to\mathbb{I}\mathbb{F}^{m\times m}$ for any $X\in\mathbb{I}\mathbb{F}^m$ satisfying
\begin{equation}\label{IntExF}
	F(X)\supset\{f(x):\forall x\in X\},
\end{equation}
and
\begin{equation}\label{IntExFp}
	F'(X)\supset\left\{f'(x):\forall x\in X\right\},
\end{equation}
respectively, where $f'(x)$ is the derivative of $f$.
A map satisfying (\ref{IntExF}) (resp. (\ref{IntExFp}))
is called an {\it interval extension} of $f\in\mathbb{R}^m\to\mathbb{R}^m$  (resp. $f'\in\mathbb{R}^m\to\mathbb{R}^{m\times m}$).
To compute $F'(X)$, we use so-called automatic differentiation, which we will explain in \S \ref{sec.Auto_diff}.

The extended Krawczyk mapping $K_F:\mathbb{I}\mathbb{F}^m\to\mathbb{I}\mathbb{F}^m$ is defined by
\[
	K_F(X):=c-RF(c)+(I-RF'(X))(X-c).
\]
Obviously $K_F(X)\supset K(X)$ holds for any $X\in\mathbb{I}\mathbb{F}^m$.
Thus, if $K_F(X)\subset\mathrm{int}(X)$, then $K(X)\subset\mathrm{int}(X)$ is satisfied.
Therefore,
if we can find computable interval extension $F$ of $f$,
which is the case of the gluing equation (\ref{considered_eq}) as mentioned above,
we can have computable Krawczyk's test $K_F(X)\subset\mathrm{int}(X)$.
Now we consider how to choose the candidate interval $X\in\mathbb{I}\mathbb{F}^m$ on which $s$ may be contractive.
Let $c\in\mathbb{F}^m$ be an approximate solution of (\ref{considered_eq}) given by a certain numerical method,
e.g. Newton's method via SnapPea.
For a given vector $x=(x_1,\cdots,x_m)\in\mathbb{F}^m$, let us define the maximum norm of $x$ by $\|x\|_{\infty}:=\max_{1\le i\le m}|x_i|$.
In our code named hikmot, we choose 
\begin{eqnarray}\label{Equation.candidate}
	X=\left(
	\begin{array}{c}
	\left[c_1-r,c_1+r\right]\\
	\left[c_2-r,c_2+r\right]\\
	\vdots\\
	\left[c_m-r,c_m+r\right]
	\end{array}
	\right),
\end{eqnarray}
as a candidate interval, where $r:=2\|Rf(c)\|_{\infty}$

\subsection{Automatic differentiation}\label{sec.Auto_diff}

%

Now, we explain how to calculate $K_F(X)$ for an interval vector $X\in\mathbb{I}\mathbb{F}^m$.
Most of the difficulty comes from the evaluation of $F'(X)$.
Therefore, we now explain how to calculate $F'(X)$ for $X\in\mathbb{I}\mathbb{F}^m$.
For calculating $F'(X)$, one may consider using 
symbolic computation.
However, the computation costs easily become too high to get useful results.
There is a reasonable alternative, a notion called the
automatic differentiation, which enables us to calculate $F'(X)$ in a mathematically rigorous way.
Although there are several ways of implementing automatic differentiation, we here explain a method called bottom up automatic differentiation.
For its implementation, we first prepare an automatic differentiation object, which is a pair 
of a data structure and set of operations among the objects in that structure.
The data structure of an automatic differentiation object is 
\[
	(d_0,d_1,d_2,\cdots,d_m)\in\mathbb{I}\mathbb{F}^{m+1}.
\]
We now define several operators on this data structure.
Let $u:\mathbb{R}\to\mathbb{R}$ be a unary operation
and $U:\mathbb{I}\mathbb{R}\to\mathbb{I}\mathbb{R}$ be its interval extension.
We assume $u$ is a differentiable map.
We denote the derivative of $u$ by $u':\mathbb{R}\to\mathbb{R}$.
Note that the derivative $u'$ we use here will be computed by hand and expressed in terms of the standard operations.
Let $U':\mathbb{I}\mathbb{R}\to\mathbb{I}\mathbb{R}$ denote an interval extension of $u'$.
Then, we define a map $\tilde{U}(p):\mathbb{I}\mathbb{F}^{m+1}\to\mathbb{I}\mathbb{F}^{m+1}$ 
which will be associate to $u$ in the bottom up automatic differentiation as a map that
maps $p=(p_0,p_1,\cdots,p_m)\in\mathbb{I}\mathbb{F}^{m+1}$
to $\tilde{U}(p)=(r_0,r_1,\cdots,r_m)\in\mathbb{I}\mathbb{F}^{m+1}$, where
\begin{align*}
	r_0&:=U(p_0),\\
	r_i&:=U'(p_0)p_i~(i=1,\cdots,m).
\end{align*}
Next, 
let $b:\mathbb{R}\times\mathbb{R}\to\mathbb{R}$ be a differentiable binary operation: $(x,y)\mapsto b(x,y)$.
The interval extension of $b$ is denoted by $B:\mathbb{I}\mathbb{R}\times\mathbb{I}\mathbb{R}\to\mathbb{I}\mathbb{R}$.
We prepare by hands partial derivatives $\partial_xb:\mathbb{R}\times\mathbb{R}\to\mathbb{R}$ and 
$\partial_yb:\mathbb{R}\times\mathbb{R}\to\mathbb{R}$ of $b$ with respect to $x$ and $y$ respectively.
For example, Table \ref{Table:autodif_operations} 
shows $\partial_xB(p_0,q_0)$ and $\partial_yB(p_0,q_0)$ for basic four arithmetic operations.
Let $\partial_xB:\mathbb{I}\mathbb{R}\times\mathbb{I}\mathbb{R}\to\mathbb{I}\mathbb{R}$
and $\partial_yB:\mathbb{I}\mathbb{R}\times\mathbb{I}\mathbb{R}\to\mathbb{I}\mathbb{R}$
denote the interval extensions of $\partial_xb$ and $\partial_yb$ respectively.
We now define a map 
$\tilde{B}:\mathbb{I}\mathbb{F}^{m+1}\times\mathbb{I}\mathbb{F}^{m+1}\to\mathbb{I}\mathbb{F}^{m+1}$, 
which will be associated to $B$ in the automatic differentiation as follows:
for $p=(p_0,p_1,\cdots,p_m),~q=(q_0,q_1,\cdots,q_m) \in\mathbb{I}\mathbb{F}^{m+1}$, 
we define
$\tilde{B}(p,q)=(r_0,r_1,\cdots,r_m)$ by
\begin{align*}
	r_0&:=B(p_0,q_0),\\
	r_i&:=\partial_xB(p_0,q_0)p_i+\partial_yB(p_0,q_0)q_i~(i=1,\cdots,m).
\end{align*}

\begin{table}[htd]
\caption{Partial derivative of basic four arithmetic operations}
\begin{center}
\begin{tabular}{c|cc}
&$\partial_xB(p_0,q_0)$&$\partial_yB(p_0,q_0)$\\
\hline
$p+q$ & $1$ & $1$\\
$p-q$ & $1$ & $-1$\\
$p\cdot q$ & $q_0$ & $p_0$\\
$p/q$ & $1/q_0$ & $-p_0/q_0^2$\\
\end{tabular}
\end{center}
\label{Table:autodif_operations}
\end{table}%

Based on the above discussion, let us now consider a map $f=(f_1,\cdots,f_m)^T$ with
$f_i:\mathbb{R}^m\to\mathbb{R}~(i=1,\cdots,m)$ being differentiable functions of $x=(x_1,\cdots,x_m)$.
We assume that for the calculation of the value $f(x)$,
we have an algorithm which consists of differentiable unary operations and differentiable binary operations.
Now we will explain how to calculate $F'(X)$ for $X=(X_1,\cdots,X_m)\in\mathbb{I}\mathbb{F}^m$.
Let $F_i:\mathbb{I}\mathbb{R}^m\to\mathbb{I}\mathbb{R}$ be the interval extension of $f_i$.
For $j=1,\cdots,m$, we denote partial derivatives of $f_i$ by $\partial_{x_j}f_i:\mathbb{R}^m\to\mathbb{R}$
whose interval extensions are defined by $\partial_{x_j}F_i:\mathbb{I}\mathbb{R}^m\to\mathbb{I}\mathbb{R}$.
We define $\tilde{X}\in\mathbb{I}\mathbb{F}^{m\times(m+1)}$ by
\[
	\tilde{X}=
	\left(
	\begin{array}{c}
	\tilde{X}_1\\
	\tilde{X}_2\\
	\vdots\\
	\tilde{X}_m
	\end{array}
	\right)
	=\left(
	\begin{array}{ccccc}
	X_1&[1,1]&[0,0]&\cdots&[0,0]\\
	X_2&[0,0]&[1,1]&\cdots&[0,0]\\
	\vdots&\vdots&\vdots&\ddots&\vdots\\
	X_m&[0,0]&[0,0]&\cdots&[1,1]
	\end{array}
	\right)
\]
for $\tilde{X}_i\in\mathbb{I}\mathbb{F}^{m+1}$.
Starting with $\tilde{X}$, replacing each operation in the algorithm for calculating the function $f$ by operations of the bottom up automatic differentiation, we have an extension of $F$ as
\[
	\tilde{F}(\tilde{X})
	=\left(
	\begin{array}{ccccc}
	F_1(X)&\partial_{x_1}F_1(X)&\partial_{x_2}F_1(X)&\cdots&\partial_{x_m}F_1(X)\\
	F_2(X)&\partial_{x_1}F_2(X)&\partial_{x_2}F_2(X)&\cdots&\partial_{x_m}F_2(X)\\
	\vdots&\vdots&\vdots&\ddots&\vdots\\
	F_m(X)&\partial_{x_1}F_m(X)&\partial_{x_2}F_m(X)&\cdots&\partial_{x_m}F_m(X)
	\end{array}
	\right)\in\mathbb{I}\mathbb{F}^{m\times(m+1)},
\]
where $\tilde{F}:\mathbb{I}\mathbb{F}^{m\times(m+1)}\to\mathbb{I}\mathbb{F}^{m\times(m+1)}$.
A submatrix of $\tilde{F}(\tilde{X})$ consists of
the second column to the last column of matrix is nothing but $F'(X)$.

\subsection{Complex arithmetic}\label{Subsection:Complex}
We will now show how to compute $\tilde{F}(\tilde{X})$ using the complex mapping $g$ in (\ref{orgequation}).
For $i=1,\cdots,n$, a data structure of complex object is defined by
\begin{equation}\label{eqn:index_rule}
	\tilde{Z}_i=\left(\tilde{X}_{2i-1},\tilde{X}_{2i}\right)\in\mathbb{I}\mathbb{F}^{(m+1)\times 2}
\end{equation}
for a pair of automatic differentiation objects $\tilde{X}_{2i-1}$ and $\tilde{X}_{2i}$.
For given $z=(z_{\mathrm{re}},z_{\mathrm{im}})$ and $w=(w_{\mathrm{re}},w_{\mathrm{im}})$,
$z_{\mathrm{re}},z_{\mathrm{im}},w_{\mathrm{re}},w_{\mathrm{im}}\in\mathbb{I}\mathbb{F}^{m+1}$,
the basic four arithmetic operations between $z$ and $w$ are defined by
\begin{align*}
	z+w&:=\left(z_{\mathrm{re}}+w_{\mathrm{re}},z_{\mathrm{im}}+w_{\mathrm{im}}\right),\\
	z-w&:=\left(z_{\mathrm{re}}-w_{\mathrm{re}},z_{\mathrm{im}}-w_{\mathrm{im}}\right),\\
	z\cdot w&:=\left(z_{\mathrm{re}}w_{\mathrm{re}}-z_{\mathrm{im}}w_{\mathrm{im}},z_{\mathrm{re}}w_{\mathrm{im}}+z_{\mathrm{im}}w_{\mathrm{re}}\right),\\
	\frac{z}{w}&:=\left(\frac{z_{\mathrm{re}}w_{\mathrm{re}}+z_{\mathrm{im}}w_{\mathrm{im}}}{w_{\mathrm{re}}^2+w_{\mathrm{im}}^2},\frac{z_{\mathrm{im}}w_{\mathrm{re}}-z_{\mathrm{re}}w_{\mathrm{im}}}{w_{\mathrm{re}}^2+w_{\mathrm{im}}^2}\right).
\end{align*}
Here, these expressions $+,-,\cdot,/$ are defined through the operations of the automatic differentiation.

Let us denote the relation (\ref{eqn:index_rule}) by
\[
	\tilde{X}=\mathrm{id}_{\mathbb{C}\to\mathbb{R}}(\tilde{Z}).
\]
or
\[
	\tilde{Z}=\mathrm{id}_{\mathbb{R}\to\mathbb{C}}(\tilde{X}).
\]

Replacing all operations for calculating the complex function $g$ by operations of the bottom up automatic differentiation mixed with complex transformation,
we can define an extension of $g$ by
$\tilde{G}:\mathbb{I}\mathbb{F}^{n\times((m+1)\times 2)}\to\mathbb{I}\mathbb{F}^{n\times((m+1)\times 2)}$.
Using $\tilde{G}$, $\tilde{F}(\tilde{X})$ can be calculated as
\[
	\tilde{F}(\tilde{X})=\mathrm{id}_{\mathbb{C}\to\mathbb{R}}\left(\tilde{G}\left(\mathrm{id}_{\mathbb{R}\to\mathbb{C}}\left(\tilde{X}\right)\right)\right).
\]

\subsection{Example}\label{exam}
As an example, we shall consider the rectangular equation for the census manifold 4\_1(5,1).
Assume that indices 
$\alpha_{j,m}$, $\beta_{j,m}$, and $\gamma_m$
are given by
\[
\{\alpha_{j,m}\}=
\left(
\begin{array}{cc}
5&9\\
2&2
\end{array}
\right),~
\{\beta_{j,m}\}=
\left(
\begin{array}{cc}
0&-7\\
-1&-1
\end{array}
\right),~
\{\gamma_m\}=
\left(
\begin{array}{c}
-1\\
1
\end{array}
\right).
\]
Then the gluing equation (\ref{orgequation}) becomes
\begin{equation}\label{eqn:4_1}
\left\{
\begin{array}{l}
z_1^5(1-z_1)^0z_2^9(1-z_2)^{-7}=-1,\\
z_1^2(1-z_1)^{-1}z_2^2(1-z_2)^{-1}=1,
\end{array}
\right.
\iff
\left\{
\begin{array}{l}
z_1^5z_2^9+(1-z_2)^7=0,\\
z_1^2z_2^2-(1-z_1)(1-z_2)=0.
\end{array}
\right.
\end{equation}
The equation (\ref{eqn:4_1}) can be written as $g(z)=0$, where $g:\mathbb{C}^2\to\mathbb{C}^2$ is defined by
\begin{align*}
	g_1(z)&=z_1^5z_2^9+(1-z_2)^7\\
	g_2(z)&=z_1^2z_2^2-(1-z_1)(1-z_2).
\end{align*}
Assuming that an approximate solution\footnote{
The decimal notation of approximate solution may include a rounding error when printing to the screen. The exact value of the approximate solution is described as
((0x1.09478e0b57659$\times 2^{-3}$)\\+(0x1.7dfbed38abdae$\times 2^{-2}$)$i$, (0x1.28cbf134a88de$\times 2^{2}$)+(0x1.afeb2e24accfd$\times 2^{0}$)$i$) in the hexadecimal.
In the rest of the paper, we note that the decimal numbers might include some rounding errors.
}
of $g(z)=0$ is given by
\[
	\tilde{z}=\left(
	\begin{array}{c}
		0.1295310113154524+0.3730313363875791i\\
		4.6374476446382840+1.6871823157824217i
	\end{array}
	\right).
\]
Our candidate interval $X\in\mathbb{I}\mathbb{F}^4$ is chosen by
\[
	X=\left(
	\begin{array}{c}
		\left[0.1295310113154227,0.1295310113154820\right]\\
		\left[0.3730313363875496,0.3730313363876089\right]\\
		\left[4.6374476446382538,4.6374476446383142\right]\\
		\left[1.6871823157823886,1.6871823157824481\right]
	\end{array}
	\right).
\]
The Krawczyk mapping of $X$ is calculated as
\[
	K_F(X)=\left(
	\begin{array}{c}
		\left[0.1295310113154520,0.1295310113154527\right]\\
		\left[0.3730313363875788,0.3730313363875796\right]\\
		\left[4.6374476446382680,4.6374476446382999\right]\\
		\left[1.6871823157824033,1.6871823157824335\right]
	\end{array}
	\right).
\]
Thus, $K_F(X)\subset\mathrm{int}(X)$ holds.
Since we obtain the sufficient condition of the computable Krawczyk's test,
the existence of a unique solution of (\ref{eqn:4_1}) in $X$ is proved.

\section{Verification package}\label{sec.verification}
In this section, we explain our package {\em hikmot} (v1.0.0) that verifies the hyperbolicity of a given manifold $M$.
Our package is available at \cite{Waseda}.
We will first explain the main algorithm, and then give instructions on how to use our package.
\subsection{The main algorithm}\label{sec.hikmot.verify}
We here explain how hikmot proves the hyperbolicity of a given triangulated manifold.
We call a triangulation with positively oriented approximated solutions by SnapPea a {\em good triangulation}.
In the parlance of SnapPea, a good triangulation corresponds to an approximation with solution type `all tetrahedra positively oriented'.
First we assume that there is a good triangulation for $M$.
A rough picture of our verification package is shown in Algorithm \ref{alg.verify}.

\begin{algorithm}
\caption{Verify hyperbolicity of $M$}
\label{alg.verify}
\begin{algorithmic}                  
\REQUIRE $M$ has a good triangulation.
\ENSURE $M$ admits hyperbolic metric of finite volume.
			\STATE Apply Krawczyk's test to the approximated SnapPea's solution for (\ref{Equation.Shapes}).
			\IF{The convergence has been verified and the imaginary parts of the solutions are all positive}
				\STATE Check arguments condition (\ref{Equation.Argument}).
				\IF{The arguments condition is also satisfied.}
					\STATE return [True, the set of intervals that contains rigorous solutions].
				\ENDIF
			\ENDIF
\STATE return [False, a collection of empty sets].
\end{algorithmic}
\end{algorithm}

We will now explain
\begin{itemize}
\item how we apply Krawczyk's test, and
\item how to check argument condition (\ref{Equation.Argument}),
\end{itemize}
in detail.
Note that we use the machine interval arithmetic for every computation.

Given a good triangulation,
we first apply Newton's method five more times to the approximated solution by SnapPea 
to get a more precise solution.
Then we apply Krawczyk's test (Theorem \ref{Theorem.Krawczyk}) to the Krawczyk mapping $K_F$ and 
the set of intervals $X$ explained in (\ref{Equation.K}) and (\ref{Equation.candidate}) respectively.
If the condition (\ref{condition.Krawczyk}), that is
$K_F(X)\subset \mathrm{int}(X)$
in Theorem \ref{Theorem.Krawczyk} holds, then we make
a list $L = [\mathrm{True},~X]$.
This means that our verification of the convergence of Newton's method have succeeded and $X$ is a set of intervals each contains the rigorous shape of a tetrahedron.
If Krawczyk's test fails, we put $L = [\mathrm{False}, E]$ where $E$ is a collection of the empty sets.

Next, we will explain how to check argument condition (\ref{Equation.Argument}).
We first assume that Krawczyk's test has succeeded, i.e. $L[0] = \mathrm{True}$.
We will use the set $X = L[1]$ of intervals.
First we note that the sum of arguments is in $\mathbb{Z}\cdot 2\pi i$.
Hence to check the argument condition (\ref{Equation.Argument}), we only need to ensure that
the sum of arguments contains only desired value.
We have prepared the function {\em atan2} whose input is an interval $I_C$ of complex numbers and
output is an interval of real numbers that contains the set $\{\arg(z)\mid z\in I_C\}$.
Essentially we have used the theory of Taylor expansion.
See appendix \ref{sec.atan2} for more detail about this atan2.
Then by using our atan2, we compute the sum of arguments 
and then check if the resulting interval contains only our desired value.
Thus we can verify that our solution satisfies condition (\ref{Equation.Argument}).

If $L[0] = \mathrm{True}$ and the argument condition (\ref{Equation.Argument}) is ensured, 
then we have a proof that $M$ is a hyperbolic 3-manifold of finite volume and
our package hikmot returns $L = [\mathrm{True}, X]$.
Otherwise, i.e. if $L[0] = \mathrm{False}$ or (\ref{Equation.Argument}) is not satisfied, then we return $L' = [\mathrm{False}, E]$.

Note that even if our verification fails for a given triangulation, it is still possible that the manifold $M$
has other good triangulations whose solutions can be verified by our package.
In practice, we 
\begin{itemize}
\item randomize triangulations of $M$ in order to get another good triangulation, and
\item try to verify the hyperbolicity by using Krawczyk's test
\end{itemize}
several times.
\begin{rem}
By using $X = L[1]$ of the set of intervals for shapes of tetrahedra, in principle, 
we can compute other invariants with rigorous error control.
\end{rem}

\subsection{Contents of our package, hikmot}
Our package hikmot contains 3 folders and 1 python files and README.txt .
First, the folder ``python\_src" contains 
\begin{itemize}
\item hikmot.py, the main file of our program. The function verify\_hyperbolicity in hikmot.py is the one that proves the hyperbolicity of a given triangulated manifold, see \S \ref{sec.hikmot.verify}.
\item interval.py, complex.py, and ftostr.py files that prepare interval arithmetic, see \S \ref{sec.Interval}. 
\item rank.py, we use this file to guess the rank of $S$ in Theorem \ref{Theorem:2_1}, see also \S \ref{sec.Interval}.
\item manifold.py, this is file has been adapted from a file originially written by Bruno Martelli \cite{MartelliPetronioRoukema}.
\end{itemize}
The folder ``cpp\_src" contains
\begin{itemize}
\item the folder ``kv", that is a set of header files(autodif.hpp, hwround.hpp, make-candidate.hpp, complex.hpp, interval-vector.hpp, interval.hpp, matrix-inversion.hpp, convert.hpp, kraw-approx.hpp, rdouble.hpp)
\item krawczyk.cc, this implements Krawczyk's test, see \S \ref{sec.Krawczyk}.
\end{itemize}
We will use setup.py to install our package, see next subsection.


\subsection{How to use the package}
In this section, we explain how to use our package on Linux or Mac machine.

We assume that SnapPy has been installed as a python module 
(see the documentation of \cite{SnapPy} for installation instructions).
We also assume boost \cite{Boost} has been installed.
Our package depends on the OS, CPU and compiler.
Therefore, users need to compile the code on the machines, which they will use.

\noindent \textbf{Installing command:}
To install hikmot as a python module, use the command as a superuser;
\begin{itemize}
\item python setup.py install
\end{itemize}
This will automatically compile our code.
We would like to thank Nathan Dunfield for writing setup.py, which eases the installation process.

\noindent\textbf{To use the code:}

\begin{enumerate}
\item Install the code.
\item Run python and import snappy and hikmot.
\item  The module hikmot.verify\_hyperbolicity($M$, print\_data) is the verifier.
	\begin{enumerate}
		\item $M$ should be a SnapPy's manifold with solution type ``all tetrahedra positively oriented".
		\item If print\_data = True, then it prints out several data.
	\end{enumerate}
\end{enumerate}

The function hikmot.verify\_hyperbolicity returns a list $L$ = [``True" or ``False", $X$ or $E$] 
where $X$ (resp. $E$) is a set intervals of tetrahedra shapes (resp. empty sets) 
as explained in section \ref{sec.verification}.


\noindent \textbf{Sample:}
On a python interpreter, we can use our codes as follows.

\begin{alltt}
\noindent\begin{itemize}
\item[] >>> import snappy
\item[] >>> import hikmot
\item[] >>> M = snappy.Manifold('4_1(5,1)')
\item[] >>> L = hikmot.verify_hyperbolicity(M)
\item[] >>> L[0]
\item[] True
\item[] >>>L[1]
\item[] {\footnotesize[([0.12953101131545199,0.12953101131545273])+([0.37303133638757879,0.37303133638757963])i, \item[] ([4.6374476446382679,4.6374476446383])+([1.6871823157824032,1.6871823157824335])i]}
\end{itemize}
\end{alltt}
\begin{alltt}
\noindent\begin{itemize}
\item[] >>> M.tetrahedra\_shapes('rect')
\item[] [(0.12953101131545247+0.3730313363875793j),
\item[]  (4.637447644638281+1.6871823157824253j)]
\end{itemize} 
\end{alltt}

\section{Applications}
\subsection{Census manifolds}
SnapPy has several censuses of manifolds (see the documentation of \cite{SnapPy} for the list of censuses).
In this section we report a computer aided verification of the hyperbolicity of manifolds in several censuses.
\begin{thm}
All the manifolds in 
OrientableCuspedCensus \cite{CaHiWe} 
are hyperbolic.
\end{thm}
\begin{proof}
We use VerifyCuspedCensus.py, (available at \cite{Waseda}) that applies our package to every manifold in OrientableCuspedCensus.
Here we summarize the result.
\begin{itemize}
\item[] \% python VerifyCuspedCensus.py, \\
Out of 17661  manifolds in the closed census, 17661  are hyperbolic and  0  remain.\\
So these manifolds remain: [] \\


\end{itemize}
On Mac OS X 10.6.8 with Intel Core 2 Duo of speed 2.13 GHz, it takes about 8 minutes. 
\end{proof}
We also verified the hyperbolicity of manifolds in Hodgson-Weeks closed census \cite{HW}.
\begin{thm}\label{thm.Closed}
All the manifolds in 
OrientableClosedCensus
are hyperbolic.
\end{thm}
\begin{proof}
For several manifolds in this SnapPy census, it is difficult to find good triangulations.
The census contains 11031 manifolds and among them 
we can find good triangulations for 10989 manifolds either directly or by randomization of triangulations.
Then our package proves the hyperbolicity of those 10989 manifolds.
For the remaining of 42 manifolds,
we will use the list, dehn.gz, which came with the older versions of SnapPea. 
The list contains several surgery descriptions for each manifold in the census and
often yields other good triangulations.
More specifically, by looking at surgery descriptions on the list dehn.gz and randomization,
among 42 manifolds, we get good triangulations and our package proves the hyperbolicity for 38 of 42, with only
[m007(3,1), m135(1,3), v3377(-1,2), v2808(-5,1)] remaining.
Here, we are using SnapPea's notation.
For example, m007(3,1) is a closed manifold obtained by filling 1-cusped manifld m007 with slope (3,1).
For the remaining four manifolds, we apply Algorithm \ref{alg.drillout} to get good triangulations.
The authors were informed the main idea of Algorithm \ref{alg.drillout} by Craig Hodgson.
Note that in practice, we also randomize triangulations at each step of Algorithm \ref{alg.drillout}.
We can find a good triangulation by this code for [m135(1,3), v3377(-1,2), v2808(-5,1)] and 
our verification package works for those triangulations and proves the hyperbolicity.
For m007(3,1), we need to take covering of degree $3$.
Since the first homology of m007 is $\mathbb{Z}\oplus\mathbb{Z}/3\mathbb{Z}$, 
there is a covering $N$ of degree $3$ with $3$ cusps.
Then by filling each cusp of $N$ by the slope (3,1), we have a degree $3$ covering $N'$ of m007(3,1).
Then by applying Algorithm \ref{alg.drillout} to $N'$, 
we can get a good triangulation.
Our package proves the hyperbolicity of $N'$ and hence m007(3,1) is hyperbolic.
\end{proof}
\begin{rem}
The randomization function on SnapPea utilizes rand() function of c language.
The function rand() depends on compiler and machine.
For the proof of Theorem \ref{thm.Closed}, we used Mac machine with OS 10.7.5.
For the repeatability, we prepared the set ClosedManifolds.zip (available at \cite{Waseda}) of closed manifolds with good triangulations, each corresponds to a manifold in OrientableClosedCensus.
One can check the hyperbolicity of those manifolds, by using VerifyClosedCensus.py, also available at \cite{Waseda}.
\end{rem}

\begin{algorithm}                      
\caption{Find positive solutions by drilling out}
\label{alg.drillout}
\begin{algorithmic}
\REQUIRE $M$ is a closed manifold with a surgery description.
\ENSURE $M$ has a good triangulation.
	\WHILE{We can find a short closed geodesic $\gamma\subset M$}
			\STATE Drill out $\gamma$ to get $M\setminus\gamma$,
			\STATE Take filled\_triangulation $N$ of $M\setminus\gamma$,
			\STATE Fill the cusp of $N$ by the slope $(1,0)$.
			\STATE (By the above procedure, we forget original surgery description and get new surgery description.)
			\IF{N has positively oriented solution.}
				\STATE return [True, $N$]
				\ENDIF
	\ENDWHILE
\STATE return False.
\end{algorithmic}
\end{algorithm}

Finally, hikmot.verify\_hyperbolicity(M,print\_data = False, save\_data=True) generates a file for M that gives all of the internal data used to compute Krawczyk's test (Theorem \ref{Theorem.Krawczyk}), which is the key step of our verification scheme. Thus, in principle, one can just use the data from this file together with an independent scheme that rigorously does computations with this data to comfirm that Krawczyk's test yields a contraction mapping, and hence verify the hyperbolicity of the manifold $M$.

\subsection{Exceptional surgeries on alternating knots}

Two of the authors, Ichihara and Masai, applied the method developed in this paper 
to study exceptional surgeries on alternating knots in the 3-sphere $S^3$. 

A \textit{knot} is an embedded circle in $S^3$, 
and it is represented by a \textit{diagram} on the plane, 
meaning that, a projected image with under-over information at each double point. 
A knot is called \textit{alternating} if it admits a diagram 
with alternatively arranged over-crossings and under-crossings running along it. 

From a given knot in $S^3$, 
by removing its tubular neighborhood and gluing solid torus back, 
one can obtain a ``new'' closed orientable 3-manifold. 
Such an operation is called a \textit{Dehn surgery} on the given knot. 
The homeomorphism type of the 3-manifold so obtained 
is determined by the isotopy class of the loop bounding a disk in the attached solid torus, 
which is called the \textit{surgery slope}. 

Due to pioneering works by Thurston \cite{ThurstonLectureNotes}, 
all but at most finitely many Dehn surgeries on a hyperbolic knot yield hyperbolic manifolds. 
Here a knot is called \textit{hyperbolic} if its complement admits 
a complete Riemannian metric of constant sectional curvature $-1$ of finite volume. 
In view of this, such an exceptional case, that is, 
a Dehn surgery on a hyperbolic knot giving a non-hyperbolic manifold, 
is called an \textit{exceptional surgery}. 
There are many of results about the classification of exceptional surgeries on knots. 
See \cite{Boyer2002} for a survey. 

In \cite{IchiharaMasai}, 
Ichihara and Masai applied our package in this paper 
to the study of exceptional surgeries on hyperbolic alternating knots, 
and achieved a complete classification. 

Here we include a rough sketch of this work. 
In theory, a result of \cite{Lackenby2000} implies that 
there are only finitely many links (i.e., disjoint unions of knots) 
so that if one could classify exceptional surgeries on those links completely, 
then a complete classification of exceptional surgeries on all hyperbolic alternating knots is obtained. 

Unfortunately the number of such links is in the millions. 
Thus the first task is to reduce the number of such links.
We now explain an outline of this step.
The links correspond to reduced alternating diagrams of alternating knots, and are filtered in terms of the complexity of the diagrams defined by Lackenby, called the \textit{twist number}. 
In the same paper, 
Lackenby proved that the alternating knots with the reduced alternating diagrams of twist number more than 8 have no exceptional surgeries. 
Further the knots with the alternating diagrams of twist number at most 5 are shown to be arborscent knots, for which the classification of exceptional surgeries is almost known. 
See \cite{IchiharaMasai} for more details. 
Therefore our target is the knots with the alternating diagrams of twist number $t$ satisfying $6 \le t \le 8$.
At this point, however, the number of the corresponding links is more than 120000. 
To reduce the number of the links further, we applied some technique using essential laminations in the link exteriors, based on the result obtained by Wu \cite{Wu2012}.
After that, we have about remaining 30000 links.

For each of the link, we apply a computer-aided approach to get a classification of exceptional surgeries on the link, essentially  developed in \cite{MartelliPetronioRoukema}. 
(As noted before, the actual code is available in the web page of B. Martelli, one of the authors of \cite{MartelliPetronioRoukema}.)
Their method is based on the so-called 6-theorem obtained 
by Agol \cite{Agol2000} and Lackenby \cite{Lackenby2000}, 
and seems very efficient in practice. 
The key point is to compute the ``length" of the surgery slope on a horotorus in hyperbolic link compliments.
Actually, the length is more than 6 implies the corresponding surgery is not exceptional due to 6-theorem.
However, at the time of writing, the procedure obtained in \cite{MartelliPetronioRoukema} mainly depends upon the Moser's work and  utilizes floating point arithmetic.
Hence we modified their code with using interval arithmetic and our package developed in this paper. 

The final problem is about the computational time. 
For each individual link, our procedure have to be applied recursively, since the links have several components. 
For example, in the worst case, for one link, we have to apply our procedure to more than 1400 cusped manifolds. 
It then takes more than 2 hours in single standard personal computer. 
Therefore we used a super-computer, called ``TSUBAME" \cite{TSUBAME1, TSUBAME.web}, set in Tokyo Institute of Technology. 
Consequently, after all computations, which took about a day in computational time in TSUBAME, we could do verify that there are no links among those we have targeted which have exceptional surgeries.


\section{Advantages over existing methods}

The algorithm described in this paper is an improvement over 
existing methods used in 3-manifold topology in three key ways: 
the control over error, the ease of extending these computations 
to other topological invariants, and the ability to implement for
large scale verifications. 
For the purposes of this 
section, we will compare this method, based on the Krawczyk test,
to an implementation of the Kontorovich test popular in the study
 of 3-manifold topology \cite{HMoser}, and finally an 
 exact arithmetic algorithm. 

The code associated to \cite{HMoser} is available via the website
in the citation. In that specific implementation, the eigenvalues of a 
conjugate transpose matrix are computed via solving a
characteristic polynomial. The eigenvalues need to be computed to high precision
because they are used to bound the size of a neighborhood of the approximate
solution that contains an exact solution.

However, an undesirable amount of precision loss can
occur during this computaiton of eigenvalues. 
It is well known that given $n 
\times n$ matrices $A=\{a_{ij}\}$ and $A_\epsilon=\{\hat{a}_{ij}\}$ such that 
$|a_{ij} - \hat{a}_{ij}|\leq \epsilon$ for all pairs, $i,j$, the minimum of
difference the eigenvalues for
$A$ and $A_\epsilon$ can be at least $\sqrt[n]{\epsilon}$. 
To be more explicit, choose $A$ to 
the matrix with ones along 
the main and submain diagonal and let 
$A_\epsilon$ be the matrix with ones along 
the main and submain diagonal and $\epsilon$ in the $(n,1)$-entry. In this 
case, the eigenvalues of $A$ are all 1 and the eigenvalues of $A_\epsilon$
are of the form $1-\zeta_n\sqrt[n]{\epsilon}$ where $\zeta_n$ is an $n$-th root of unity. 

Despite this potential for precision loss in \cite{HMoser} as described above,
it is doubtful that a small complexity 3-manifold exploits such an error in
order to a false certificate of hyperbolicity via that computation.  
More concretely, we point out that every manifold verified in 
\cite{HMoser} has also been verified by the methods of this paper 
(including the large links of Leininger \cite{Leininger},
 which have 32 and 44 tetrahedra). 


 %

The second feature of the interval arithmetic technique is
that the computer keeps track of the accumulated error throughout the computation.
Therefore, the same techniques of computation allow for further rigorously 
 computed invariants coming from the solutions to the 
 gluing equations. In principle, any computation that Snappea preforms can be made rigorous using
 the techniques of this paper.
 
Finally, a third method exists for rigorous computer verification of hyperbolicity,
 namely snap, which uses an exact arithmetic algorithm
(see \cite{snap, snapArt}). Snap 
 is able to verify the exact hyperbolic structure
by representing algebraic numbers exactly and algebraically solve the polynomial equations 
of a manifold (\ref{Equation.EfficientSetPolynomial}). With this data, the user can compute 
arithmetic data related to hyperbolic 3-manifolds (which our methods cannot).
However, the exact arithmetic methods are designed principally for the computations of this 
arithmetic data, and so the methods of this paper are more efficient
to rigorously verify hyperbolicity. 
For instance, using the snap software would be less practical than our methods for a large scale 
verification involving large sets of manifolds with relatively high numbers of tetrahedra
such as the computation preformed in \cite{IchiharaMasai}. In addition, snap does not report the error on its ``non-arithmetic" computations,  such as volume or parabolic length. Therefore, a separate argument addressing such error would have to be derived one is able to claim rigorous computations using data from snap.

\appendix
\section{Verified computations for $\arg(z)$}\label{sec.atan2}
In this appendix, we describe how to rigorously compute $\arg(z)$ for $z=x+iy\in \mathbb{C}$ using the floating-point arithmetic.
The complex argument of $z$ is defined by
\[
	\arg(z):=\mathrm{atan2}(y,x).
\]
Here, $\mathrm{atan2}(y,x)$ is one of
commonly used mathematical functions
in programming languages including C, C++, JAVA etc.,
and defined by Table \ref{atan2_real}.
\begin{table}[htb]
\caption{Definition of $\mathrm{atan2}(y,x)$}
\begin{center}
  \begin{tabular}{c|c} \hline
	Conditions for $x$ and $y$&$\mathrm{atan2}(y,x)$\\\hline\\[-3mm]
    $y \le x, \quad y > -x$ & ${\arctan}(y/x)$ \\[2mm]
    $y > x, \quad y > -x$ & $\frac{\pi}{2} - {\arctan}(x/y)$ \\[2mm]
    $y > x, \quad y \le -x, \quad y \ge 0$ & $\pi + {\arctan}(y/x)$ \\[2mm]
    $y > x, \quad y \le -x, \quad y < 0$ & $-\pi + {\arctan}(y/x)$ \\[2mm]
    $y \le x, \quad y \le -x$ & $-\frac{\pi}{2} - {\arctan}(x/y)$ \\[2mm]\hline
  \end{tabular}
\end{center}
\label{atan2_real}
\end{table}
In this table, the $\arctan$ function is assumed to have the range $(-\frac{\pi}{2},\frac{\pi}{2})$.
Here, we first show how to rigorously evaluate an interval extension of $\arctan$ function.
Then, based on this we present a rigorous method of evaluating $\mathrm{atan2}$ function.

\subsection{Arctangent}
First, we show how to construct an interval extension of $\arctan$ function.
It is seen from Table \ref{atan2_real} that the evaluation of $\arctan(x)$ for $x\in\mathbb{F}$
reduces to that for $\arctan$ function on the interval $[-(\sqrt{2} - 1),\sqrt{2} - 1]$.
On the interval $\{x:|x| \le \sqrt{2} - 1\}$,
an interval inclusion of $\arctan(x)$ is obtained from the Maclaurin expansion with remainder term:
\begin{equation}\label{atan2_inc}
  \arctan(x) \in x - \frac{1}{3}x^3 + \frac{1}{5}x^5 - \frac{1}{7}x^7 + \cdots + \frac{1}{n}[-1,1]x^n
\end{equation}
with $n$ being chosen as
\[
    \left|\frac{1}{n}[-1,1]x^n\right|\le 2^{-53}.
\]
Calculating the right-hand side of (\ref{atan2_inc}) with interval arithmetic,
we define $\mathrm{atan\_point}:\mathbb{F}\to\mathbb{I}\mathbb{F}$ by
\[
	\mathrm{atan\_point}(x):=x - \frac{1}{3}x^3 + \frac{1}{5}x^5 - \frac{1}{7}x^7 + \cdots + \frac{1}{n}[-1,1]x^n.
\]
For $I=[a,b]\in\mathbb{I}\mathbb{F}$, we define $\underline{I}=a$ and $\overline{I}=b$ respectively.
Then, an interval extension $\mathrm{Arctan}:\mathbb{I}\mathbb{F}\to\mathbb{I}\mathbb{F}$
is given by
\[
	\mathrm{Arctan}(X):=\left[ \underline{\mathrm{atan\_point}(\underline{X})}, \overline{\mathrm{atan\_point}(\overline{X})}\right]\in \mathbb{I}\mathbb{F}.
\]


\subsection{Evaluation of $\arg(z)$}
Next, we show how to construct an interval extension of $\mathrm{atan2}$ function whose range is assumed to be $(-\pi,\pi]$.
For $x,y\in \mathbb{F}$, let us denote $I_{y/x}:=[\mathrm{fldown}(y/x),\mathrm{flup}(y/x)]$
and $I_{x/y}:=[\mathrm{fldown}(x/y),\mathrm{flup}(x/y)]$.
Furthermore, $[\frac{\pi}{2}]$ and $[\pi]$ denote the interval inclusion of $\frac{\pi}{2}$ and $\pi$ respectively.
An interval inclusion of $\mathrm{atan2}(y,x)$, which is denoted by $\mathrm{atan2\_point}(y,x)$,
can be calculated through $\mathrm{Arctan}$ function defined above.
Namely, Table \ref{atan2_point} shows a realization:
\begin{table}[htb]
\caption{Realization of $\mathrm{atan2\_point}(y,x)$}
\begin{center}
  \begin{tabular}{c|c} \hline
	Conditions for $x$ and $y$&$\mathrm{atan2\_point}(y,x)$\\\hline\\[-3mm]
    $y \le x, \quad y > -x$ & $\mathrm{Arctan}(I_{y/x})$ \\[2mm]
    $y > x, \quad y > -x$ & $[\frac{\pi}{2}] - \mathrm{Arctan}(I_{x/y})$ \\[2mm]
    $y > x, \quad y \le -x, \quad y \ge 0$ & $[\pi] + \mathrm{Arctan}(I_{y/x})$ \\[2mm]
    $y > x, \quad y \le -x, \quad y < 0$ & $-[\pi] + \mathrm{Arctan}(I_{y/x})$ \\[2mm]
    $y \le x, \quad y \le -x$ & $-[\frac{\pi}{2}] - \mathrm{Arctan}(I_{x/y})$ \\[2mm]\hline
  \end{tabular}
\end{center}
\label{atan2_point}
\end{table}

For $I_y$, $I_x\in \mathbb{I}\mathbb{F}$,
Table \ref{atan2} shows  a realization of $\mathrm{Atan2}:\mathbb{I}\mathbb{F}\times\mathbb{I}\mathbb{F}\to\mathbb{I}\mathbb{F}$,
which is a kind of interval extension of $\mathrm{atan2}$ function.
It is easily seen from Table \ref{atan2},
$\mathrm{Atan2}$ function is an interval extension of $\mathrm{atan2}$ function except the case of
$I_x \not\ni 0$, $I_y \ni 0$, and $\overline{I_x} < 0$.
In this case, a natural interval extension of $\mathrm{atan2}$ function becomes the union of the following two intervals:
\[
	\left[-\overline{[\pi]},\underline{\mathrm{atan2\_point}(\overline{I_y},\overline{I_x})}\right]~\mbox{and}~\left[\overline{\mathrm{atan2\_point}(\underline{I_y},\overline{I_x})},\overline{[\pi]}\right].
\]
Since this makes the algorithm multivalued, 
we modify the algorithm to return a single interval for such $I_y$ and $I_x$:
\[
	\mathrm{Atan2}(I_y,I_x)=\left[\underline{\mathrm{atan2\_point}(\overline{I_y},\overline{I_x})},~\overline{2[\pi] + \mathrm{atan2\_point}(\underline{I_y},\overline{I_x})}\right].
\]
It is just a change of expression so that this modification does not cause any confusion.

\begin{table}[htb]
\caption{$\mathrm{Atan2}(I_y,I_x)$}
\begin{center}
  \begin{tabular}{c|c} \hline
	Conditions for $I_x$ and $I_y$&$\mathrm{Atan2}(I_y,I_x)$ \\\hline\\[-3mm]
   $I_x \ni 0, \quad I_y \ni 0$ & $\left[-\overline{[\pi]},\overline{[\pi]}\right]$ \\[2mm]
   $I_x \ni 0, \quad I_y \not\ni 0, \quad \underline{I_y} > 0$
     & $\left[\underline{\mathrm{atan2\_point}(\underline{I_y},\overline{I_x})},
       \overline{\mathrm{atan2\_point}(\underline{I_y},\underline{I_x})}\right]$
     \\[4mm]
   $I_x \ni 0, \quad I_y \not\ni 0, \quad \overline{I_y} < 0$
     & $\left[\underline{\mathrm{atan2\_point}(\overline{I_y},\underline{I_x})},
       \overline{\mathrm{atan2\_point}(\overline{I_y},\overline{I_x})}\right]$
     \\[4mm]
   $I_x \not\ni 0, \quad I_y \ni 0, \quad \underline{I_x} > 0$
     & $\left[\underline{\mathrm{atan2\_point}(\underline{I_y},\underline{I_x})},
       \overline{\mathrm{atan2\_point}(\overline{I_y},\underline{I_x})}\right]$
     \\[4mm]
   $I_x \not\ni 0, \quad I_y \ni 0, \quad \overline{I_x} < 0$
     & $\left[\underline{\mathrm{atan2\_point}(\overline{I_y},\overline{I_x})},
       \overline{2[\pi] + \mathrm{atan2\_point}(\underline{I_y},\overline{I_x})}\right]$
     \\[4mm]
   $I_x \not\ni 0, \quad I_y \not\ni 0, \quad \underline{I_x} > 0, \quad \underline{I_y} > 0$
     & $\left[\underline{\mathrm{atan2\_point}(\underline{I_y},\overline{I_x})},
       \overline{\mathrm{atan2\_point}(\overline{I_y},\underline{I_x})}\right]$
     \\[4mm]
   $I_x \not\ni 0, \quad I_y \not\ni 0, \quad \underline{I_x} > 0, \quad \overline{I_y} < 0$
     & $\left[\underline{\mathrm{atan2\_point}(\underline{I_y},\underline{I_x})},
       \overline{\mathrm{atan2\_point}(\overline{I_y},\overline{I_x})}\right]$
     \\[4mm]
   $I_x \not\ni 0, \quad I_y \not\ni 0, \quad \overline{I_x} < 0, \quad \underline{I_y} > 0$
     & $\left[\underline{\mathrm{atan2\_point}(\overline{I_y},\overline{I_x})},
       \overline{\mathrm{atan2\_point}(\underline{I_y},\underline{I_x})}\right]$
     \\[4mm]
   $I_x \not\ni 0, \quad I_y \not\ni 0, \quad \overline{I_x} < 0, \quad \overline{I_y} < 0$
     & $\left[\underline{\mathrm{atan2\_point}(\overline{I_y},\underline{I_x})},
       \overline{\mathrm{atan2\_point}(\underline{I_y},\overline{I_x})}\right]$
     \\[2mm]\hline
  \end{tabular}
\end{center}
\label{atan2}
\end{table}

%
%
%
%

\end{document}